\documentclass[12pt]{article}
\usepackage{amsbsy,amsmath,amsthm,amssymb,subcaption}
\usepackage{graphicx}
\usepackage{enumerate}
\usepackage{natbib}
\usepackage{url} 
\usepackage{hyperref}
\usepackage{bm}
\usepackage{xr}
\usepackage{alphalph}

\newcommand{\blind}{0}

\makeatletter
\def\namedlabel#1#2{\begingroup
	#2%
	\def\@currentlabel{#2}%
	\phantomsection\label{#1}\endgroup
}
\makeatother

\newcommand{\df}{\mathrm{d}}
\newcommand{\X}{\mathsf{X}}
\newcommand{\Y}{\mathsf{Y}}
\newcommand{\Z}{\mathsf{Z}}
\newcommand{\SF}{\mathcal{A}}
\newcommand{\A}{\mathcal{A}}
\newcommand{\F}{\mathcal{F}}

\newcommand{\ind}{\mathbf{1}}
\newcommand{\Mtk}{\mtkfont{T}}
\newcommand{\mtkfont}{\mathcal}
\newcommand{\pcite}[1]{\citeauthor{#1}'s \citeyearpar{#1}}

\newtheorem{theorem}{Theorem}
\newtheorem{lemma}[theorem]{Lemma}

\newtheorem{proposition}[theorem]{Proposition}

\newtheorem{remark}[theorem]{Remark}
\usepackage{xcolor}

\begin{document}

\def\spacingset#1{\renewcommand{\baselinestretch}%
{#1}\small\normalsize} \spacingset{1}


\if1\blind
{
  \title{\bf Geometric ergodicity of trans-dimensional Markov chain Monte Carlo algorithms}
  \author{}
  \maketitle
} \fi

\if0\blind
{
  \title{\bf Geometric ergodicity of trans-dimensional Markov chain Monte Carlo algorithms}
  \author{Qian Qin\thanks{
  		Partially supported by NSF DMS-2112887}\hspace{.2cm}\\
  	School of Statistics, University of Minnesota}
  \maketitle
} \fi

\bigskip
\begin{abstract}
This article studies the convergence properties of trans-dimensional MCMC algorithms when the total number of models is finite.
It is shown that, for reversible and some non-reversible trans-dimensional Markov chains, under mild conditions, geometric convergence is guaranteed if the Markov chains associated with the within-model moves are geometrically ergodic.
This result is proved in an $L^2$ framework using the technique of Markov chain decomposition.
While the technique was previously developed for reversible chains, this work extends it to the point that it can be applied to some commonly used non-reversible chains.
The theory herein is applied to reversible jump algorithms for three Bayesian models: a probit regression with variable selection, a Gaussian mixture model with unknown number of components, and an autoregression with Laplace errors and unknown model order.
\end{abstract}

\noindent%
\vfill

\newpage
\spacingset{1.9} 

	\section{Introduction} \label{sec:intro}

In many statistical setups, the parameter space of interest is a union of disjoint subsets, where each subset corresponds to a model, and the dimensions of the subsets need not be the same.
Trans-dimensional Markov chain Monte Carlo (MCMC) is a class of algorithms for sampling from distributions defined on such spaces, which allows for model selection as well as parameter estimation.
This type of algorithm, especially the reversible jump MCMC developed by \cite{green1995reversible}, has been applied to important problems like change-point estimation \citep{green1995reversible}, autoregression models \citep{troughton1998reversible, ehlers2002efficient, vermaak2004reversible}, variable selection \citep{chevallier2022reversible}, wavelet models \citep{cornish2015bayeswave} etc.
The current article aims to provide conditions on geometric ergodicity for trans-dimensional Markov chains when the total number of models is finite.

Let $\mathsf{K}$ be a finite set whose elements are referred to as ``models."
A model~$k$ in~$\mathsf{K}$ is associated with a non-empty measurable space $(\Z_k, \A_k)$ and a nonzero finite measure $\Psi_k$ on $(\Z_k, \A_k)$.
Let $\X = \bigcup_{k \in \mathsf{K}} \{k\} \times \Z_k$, and let $\SF$ be the sigma algebra generated by sets of the form $\{k\} \times \mathsf{A}$, where $k \in \mathsf{K}$ and $\mathsf{A} \in \A_k$.
Consider the task of sampling from the probability measure~$\Pi$ on $(\X, \SF)$ such that
\begin{equation} \label{eq:Pi}
	\Pi(\{k\} \times \mathsf{A}) = \frac{\Psi_k(\mathsf{A})}{ \sum_{k' \in \mathsf{K}} \Psi_{k'}(\Z_{k'}) }, \quad k \in \mathsf{K}, \; \mathsf{A} \in \A_k.
\end{equation}
Suppose that a procedure generates a random element $(K,Z) \sim \Pi$.
Then, for $k \in \mathsf{K}$, $\Psi_k(\Z_k)/ \sum_{k' \in \mathsf{K}} \Psi_{k'}(\Z_{k'})$ gives the probability of $K = k$, and $\Phi_k(\cdot) = \Psi_k(\cdot)/\Psi_k(\Z_k)$ gives the conditional distribution of~$Z$ given $K = k$.

In practice,~$\Pi$ is often intractable, prompting the use of trans-dimensional MCMC methods.
A central goal of the current work is to provide verifiable sufficient conditions for trans-dimensional MCMC algorithms to be geometrically convergent in the $L^2$ distance.
When~$\SF$ is countably generated and the Markov chain is $\varphi$-irreducible, $L^2$ geometric convergence implies the classical notion of $\Pi$-a.e. geometric ergdoicity.
See \cite{roberts1997geometric}, \cite{roberts2001geometric}, and \cite{gallegos2023equivalences}.
Geometric ergodicity is one of the key conditions ensuring the reliability of MCMC estimation.
It guarantees a central limit theorem (CLT) for ergodic sums \citep{jones2001honest}; moreover, consistent uncertainty assessment through asymptotically valid confidence intervals is possible under geometric ergodicity \citep{vats2019multivariate}.


The convergence behavior of trans-dimensional MCMC algorithms is in general far from well understood.
\cite{roberts2006harris} established some general conditions for trans-dimensional Markov chains to be Harris recurrent.
Geometric ergodicity of some specific trans-dimensional algorithms was established in \cite{geyer1994simulation}, \cite{andrieu1999joint}, \cite{ortner2006reversible}, and \cite{schreck2015shrinkage}.
Existing proofs of geometric ergodicity often rely on drift and minorization conditions, or in some simple situations, Doeblin's condition.
The current work instead utilizes the decomposition of Markov chains, a remarkable technique developed by \cite{caracciolo1992two} and documented in \cite{madras2002markov}.
This technique allows one to decompose the dynamic of a trans-dimensional Markov chain into within- and between-model movements, which can be analyzed separately.
Using an extended version of this technique and exploiting the assumption that $|\mathsf{K}| < \infty$, Theorem~\ref{thm:main} is established.
This result describes a divide-and-conquer paradigm that enables one to establish geometric convergence of the trans-dimensional chain by combining the geometric ergodicity of its within-model components.
Quantitative bounds on the convergence rate will also be provided; see Theorem~\ref{thm:quantitative}.

Previously, Markov chain decomposition has found its use in important problems like simulated and parallel tempering.
See, e.g., \cite{madras2002markov, woodard2009conditions, ge2018simulated}.
This technique can be used to analyze a Markov chain whose state space can be partitioned into subsets such that, within each subset, the Markov chain's behavior is easy to analyze.
It was originally developed for reversible Markov chains.
The current work provides an extended version of the technique, given in Lemma~\ref{lem:decomposition}, that can deal with some important non-reversible chains.
(Despite their name, reversible jump algorithms can often be non-reversible.)

The theory developed herein is applied to reversible jump MCMC algorithms for three practical Bayesian models:
a probit regression model with variable selection, a Gaussian mixture model with unknown number of components, and an autoregressive model with Laplace errors and unknown model order.
Among these algorithms, only one is assuredly reversible. 
For each algorithm, we demonstrate geometric convergence. 
Enabled by geometric ergodicity, we also conduct Monte Carlo uncertainty assessments.


Finally, it must be emphasized that verifying geometric ergodicity is but one of the first steps towards fully understanding the convergence behavior of an MCMC algorithm.
A chain being geometrically convergent does not ensure that it has a fast convergence rate.
While this work does provide a quantitative convergence rate bound, calculating the quantities involved in the bound can be practically challenging.
Obtaining sharp estimates for the convergence rate remains an open problem for most practical trans-dimensional MCMC algorithms.

The rest of this article is organized as follows.
Following a quick overview of the main qualitative result of this article, Section~\ref{sec:l2} contains some preliminary facts on the $L^2$ theory of Markov chains.
The main technical results involving Markov chain decomposition and the convergence rate of trans-dimensional MCMC are given in Section~\ref{sec:main}.
One toy and two practical examples are studied in Section~\ref{sec:examples}, followed by a brief discussion in Section \ref{sec:discussion}.
\ref{s1} contains some minor results and technical proofs.
\ref{s2} contains yet another practical example.

\subsection{Conditions for geometric ergodicity: An overview}

Consider a trans-dimensional Markov chain $(X(t))_{t=0}^{\infty} = (K(t), Z(t))_{t=0}^{\infty}$ whose state space is~$\X$.
Let $P: \X \times \SF \to [0,1]$ be its Markov transition kernel (Mtk), i.e., for $(k,z) \in \X$ and $\mathsf{A} \in \SF$, $P((k,z), \mathsf{A})$ is understood as the conditional probability of $(K(t+1), Z(t+1)) \in \mathsf{A}$ given $(K(t), Z(t)) = (k,z)$.
Suppose that~$\Pi$ is a stationary distribution of this chain, i.e., $\Pi = \Pi P$, or more explicitly, for $\bigcup_{k \in \mathsf{K}} \{k\} \times \mathsf{A}_k \in \SF$,
\[
\sum_{k' \in \mathsf{K}} \Psi_{k'}(\mathsf{A}_{k'}) = \sum_{k \in \mathsf{K}} \sum_{k' \in \mathsf{K}} \int_{\Z_k} \Psi_k(\df z) P((k,z),  \{k'\} \times \mathsf{A}_{k'}).
\]


The main results of this paper are stated in terms the $L^2$ theory for Markov chains, which is reviewed in Section~\ref{sec:l2}.
Essentially, if $\Mtk(\cdot,\cdot)$ is an Mtk that has a stationary distribution~$\omega$, then~$\Mtk$ can be regarded as a bounded linear operator on a certain Hilbert space $L_0^2(\omega)$.
A sufficient condition for the corresponding chain to be $L^2$ geometrically convergent is that the operator norm of some power of~$\Mtk$ is less than one.

One of the main results of this paper is stated below.
See Section~\ref{sec:main} for more details.

\begin{theorem} \label{thm:main}
	Assume that each of the following conditions holds for the trans-dimensional chain:
	\begin{enumerate} 
		\item [\namedlabel{H1}{(H1)}] There are a positive integer~$t_0$ and a sequence of Mtks $P_k: \Z_k \times \A_k \to [0,1], \; k \in \mathsf{K},$ such that the following properties hold for each~$k$:
		\begin{enumerate}
			\item [(i)] $\Phi_k P_k = \Phi_k$, where $\Phi_k(\cdot) = \Psi_k(\cdot)/\Psi_k(\Z_k)$ is the normalization of $\Psi_k(\cdot)$.
			\item [(ii)] When $P_k$ is regarded as an operator on $L_0^2(\Phi_k)$, the norm of its $t_0$'th power $P_k^{t_0}$ is strictly less than~1.
			\item [(iii)] there exists a constant $c_k > 0$ such that 
			$
			P((k,z), \{k\} \times \mathsf{A}) \geq c_k P_k(z, \mathsf{A})
			$ for $z \in \Z_k$ and $\mathsf{A} \in \A_k$.
		\end{enumerate} 
		
		\item [\namedlabel{H2}{(H2)}] 
		{ The between-model movements are irreducible.
		To be precise, the Mtk $\bar{P}: \mathsf{K} \times 2^{\mathsf{K}} \to [0,1]$ given by $
			\bar{P}(k,\{k'\}) = \int_{\Z_k} \Phi_k(\df z) P((k,z) , \{k'\} \times \Z_{k'}), \; k, k' \in \mathsf{K},
			$
		is irreducible.}
	\end{enumerate}
	Then the norm of $P^{t_0}$ is strictly less than one, and the trans-dimensional chain is $L^2(\Pi)$ geometrically convergent.
\end{theorem}

Note that $\bar{P}(k,\{k'\})$ can be understood as the average probability flow from model~$k$ to model~$k'$, and thus $\bar{P}$ characterizes the between movements.
{ Evidently, \ref{H2} holds as long as the chain is $\Pi$-irreducible.
We give a rigorous proof of this assertion in Section \ref{app:h2} of \ref{s1}.}
Hence, in practice, \ref{H2} usually trivially holds.

Trans-dimensional MCMC algorithms typically involve a within-model move type, where the underlying chain stays in a model, say~$k$, with probability~$c_k$, and move according to an Mtk $P_k$ such that $\Phi_k P_k = \Phi_k$.
Then such $c_k$ and $P_k$ satisfy (i) and (iii) in \ref{H1}.
Condition (ii) in \ref{H1} requires a careful analysis of the within-model moves of an algorithm.
According to Lemma~\ref{lem:Pk-convergence} below, in several important situations, this condition is implied by the geometric ergodicity of the chain associated with $P_k$, or some closely related Markov chain whose state space is $\Z_k$.
Since the space $\Z_k$ is typically of a fixed dimension, the hope is that chains that move in $\Z_k$ can be analyzed using well-established tools such as drift and minorization or functional inequalities.

\begin{lemma} \label{lem:Pk-convergence}
	Let~$k$ be in $\mathsf{K}$.
	Suppose that $\SF_k$ is countably generated, $\Phi_k P_k = \Phi_k$, and the chain associated with $P_k$ is $\varphi$-irreducible.
	Then, in each of the following situations, (ii) in \ref{H1} holds with $t_0 = 1$.
	\begin{enumerate}
		\item [(i)] $P_k$ defines a $\Phi_k$-a.e. geometrically ergodic chain that is reversible with respect to $\Phi_k$.
		\item [(ii)] $P_k$ defines a deterministic-scan Gibbs chain with two components that is $\Phi_k$-a.e. geometrically ergodic.
		\item [(iii)] $P_k$ defines a deterministic-scan Gibbs chain, and there exists a $\Phi_k$-a.e. geometrically ergodic random-scan Gibbs chain based on the same set of conditional distributions.
	\end{enumerate}
\end{lemma}
\begin{proof}
	For (i), see Theorem 2.1 of \cite{roberts1997geometric} and Theorem 2 of \cite{roberts2001geometric}.
	For (ii), see Lemma 3.2 and Proposition 3.5 of \cite{qin2020convergence}.
	For (iii), see Theorem 3.1 of \cite{chlebicka2023solidarity} and invoke (i).
\end{proof}

In Section \ref{ssec:mixture}, we give an example of establishing and utilizing (ii) in \ref{H1} with $t_0 > 1$.

\section{Preliminaries}  \label{sec:l2}

Let $(\Y, \F, \omega)$ be a generic probability space.
Let $L^2(\omega)$ be the Hilbert space of real functions $f: \Y \to \mathbb{R}$ that are square integrable with respect to~$\omega$, with the inner product between two functions defined as
$
\langle f, g \rangle_{\omega} = \int_{\Y} f(y) g(y) \, \omega(\df y),
$
and the norm defined as $\|f\|_{\omega} = \sqrt{\langle f, f \rangle_{\omega}}$.
Denote by $L_0^2(\omega)$ the subspace of $L^2(\omega)$ that consists of functions~$f$ such that $\omega f := \langle f, \ind_{\Y} \rangle_{\omega} = 0$, where $\ind_{\Y}(y) = 1$ for $y \in \Y$.
A probability measure $\mu$ on $(\Y, \F)$ is said to be in $L_*^2(\omega)$ if $\df \mu/ \df \omega$ exists and is in $L^2(\omega)$. 
For two probability measures $\mu$ and $\nu$ in $L_*^2(\omega)$, define their $L^2$ distance by $\|\mu - \nu\|_{\omega} = \|\df\mu/\df \omega - \df \nu/\df \omega\|_{\omega}$.

Let $\Mtk: \Y \times \F \to [0,1]$ be an Mtk whose stationary distribution is~$\omega$.
For a probability measure~$\mu$ on~$\F$, define $\mu \Mtk^t(\cdot) = \int_{\Y} \mu(\df y) \Mtk^t(y, \cdot)$, where $\Mtk^t$ is the corresponding $t$-step Mtk.
We say $\Mtk$ is $L^2(\omega)$ geometrically convergent if there exist $\rho < 1$ and a function $C: L_*^2(\omega) \to [0,\infty)$ such that for $\mu \in L_*^2(\omega)$ and $t \geq 1$,
\begin{equation} \label{ine:geometric}
	\|\mu \Mtk^t - \omega\|_{\omega} \leq C(\mu) \rho^t.
\end{equation}
Let $\|\cdot\|_{\scriptsize\mbox{TV}}$ be the total variance distance between two probability measures.
$\Mtk$ is said to be $\omega$-a.e. geometrically ergodic if there exist $\rho < 1$ and $C: \Y \to [0, \infty)$ such that, for $\omega$-almost every $y \in \Y$ and $t \geq 1$,
$
	\|\Mtk^t(y,\cdot) - \omega(\cdot) \|_{\scriptsize\mbox{TV}} \leq C(y) \rho^t.
$
Results from \cite{roberts2001geometric} indicate that when $\F$ is countably generated, if the chain is $L^2(\omega)$ geometrically convergent, then it is $\omega$-a.e. geometrically ergodic; the converse holds if $\Mtk$ is reversible with respect to~$\omega$.
See also \cite{roberts1997geometric} and \cite{gallegos2023equivalences}.

The Mtk~$\Mtk$ can be understood as a linear operator on $L_0^2(\omega)$: for $f \in L_0^2(\omega)$, 
$
\Mtk f(\cdot) = \int_{\Y} \Mtk(\cdot, \df y) f(y).
$
One can use the Cauchy-Schwarz inequality to show that the $L^2$ norm of $\Mtk$, defined as
\[
\|\Mtk\|_{\omega} = \sup_{f \in L_0^2(\omega) \setminus \{0\}} \frac{\|\Mtk f\|_{\omega}}{\|f\|_{\omega}},
\]
is no greater than~1.
The operator norms of $\Mtk$ and its powers quantify the Markov chain's convergence rate, with smaller norms indicating faster convergence.
Indeed, if $s$ is a positive integer, then~\eqref{ine:geometric} holds for all $\mu \in L_*^2(\omega)$ and $t \geq 1$ with $\rho = \|\Mtk ^s\|_{\omega}^{1/s}$ and some $C(\cdot)$.
See Theorem 2.1 of \cite{roberts1997geometric} for more details on the interpretation of $\|\Mtk\|_{\omega}$.

The bounded operator~$\Mtk $ has a unique adjoint $\Mtk ^*$.
It is well-known that $\|\Mtk \|_{\omega}^2 = \|\Mtk ^*\|_{\omega}^2 = \|\Mtk  \Mtk ^*\|_{\omega} = \|\Mtk ^*\Mtk \|_{\omega}$.
The Mtk $\Mtk$ is reversible with respect to $\omega$ if and only if the operator $\Mtk$ is self-adjoint, i.e., $\Mtk  = \Mtk ^*$.
The operator~$\Mtk $ is positive semi-definite if it is self-adjoint, and $\langle \Mtk f, f \rangle_{\omega} \geq 0$ for $f \in L_0^2(\omega)$.

When $\Mtk$ is self-adjoint, its spectral gap is defined to be
\[
\mbox{Gap}_{\omega} (\Mtk) = 1 - \sup_{f \in L_0^2(\omega) \setminus \{0\}} \frac{\langle f, \Mtk f \rangle_{\omega}}{\|f\|_{\omega}^2}.
\]
Note that $	\mbox{Gap}_{\omega} (\Mtk) \geq 1 - \|\Mtk\|_{\omega} \geq 0$.
If $\Mtk$ is positive semi-definite, $\|\Mtk \|_{\omega} = 1 - \mbox{Gap}_{\omega}(\Mtk)$.

\section{Convergence Analysis} \label{sec:main}

\subsection{Markov chain decomposition}

This subsection describes the main probabilistic tool for proving Theorem~\ref{thm:main}.

Again, let $(\Y, \F, \omega)$ be a probability space.
Suppose that $\Y$ can be partitioned into a collection of disjoint subsets, $(\Y_k)_{k \in \mathsf{K}}$.
For this subsection, allow~$\mathsf{K}$ to be countably infinite.
Assume that $\omega(\Y_k) > 0$ for each~$k$.
\cite{caracciolo1992two} proposed a framework for analyzing a Markov chain moving in~$\Y$ by decomposing its dynamic into local movements within a subset $\Y_k$ and global movements across the disjoint subsets.
The key technical result, published in \cite{madras2002markov}, is stated for reversible chains \citep[see also, e.g.,][]{guan2007small,woodard2009conditions}.
Here, it is extended to a possibly non-reversible setting.

For $k \in \mathsf{K}$, let $\F_k$ be the restriction of~$\F$ on $\Y_k$, and let $\omega_k(\mathsf{B}) = \omega(\mathsf{B})/\omega(\Y_k)$ for $\mathsf{B} \in \F_k$.
Let $\bar{\omega}(\{k\}) = \omega(\Y_k)$ for $k \in \mathsf{K}$.
{ For an Mtk $\mtkfont{S}: \Y \times \F \to [0,1]$ such that $\omega \mtkfont{S} = \omega$, let
$\bar{\mtkfont{S}}$ be an Mtk on the discrete space $\mathsf{K}$ such that, for $k, k' \in \mathsf{K}$, 
\[
\bar{\mtkfont{S}}(k, \{k'\}) = \frac{1}{\omega(\Y_k)} \langle \ind_{\Y_k}, S \ind_{\Y_{k'}} \rangle_{\omega} = \frac{1}{\omega(\Y_k)} \int_{\Y_k} \omega(\df y) \, \mtkfont{S}(y, \Y_{k'}).
\] 
Then $\bar{\omega} \bar{\mtkfont{S}} = \bar{\omega}$.
It can be checked that $\bar{\mtkfont{S}}$ defines a self-adjoint (resp. positive semi-definite) operator on~$L_0^2(\bar{\omega})$ whenever $\mtkfont{S}$ is self-adjoint (resp. positive semi-definite). }
In the same vein, define the Mtk $\overline{\mtkfont{S}^*\mtkfont{S}}$ on $\mathsf{K}$, which takes the form
\[
\overline{\mtkfont{S}^*\mtkfont{S}}(k,\{k'\}) = \frac{1}{\omega(\Y_k)} \langle \ind_{\Y_k}, S^* S \ind_{\Y_{k'}} \rangle_{\omega} = \frac{1}{\omega(\Y_k)} \int_{\Y} \omega(\df y) \, \mtkfont{S}(y, \Y_k) \mtkfont{S}(y, \Y_{k'}).
\]
As long as $\omega \mtkfont{S} = \omega$, $\overline{\mtkfont{S}^*\mtkfont{S}}$ defines a positive semi-definite operator on~$L_0^2(\bar{\omega})$.

Below is the key technical lemma of this subsection. 

\begin{lemma} \label{lem:decomposition}
	Let $\Mtk $ and $\mtkfont{S}$ be Mtks such that $\omega \Mtk  = \omega \mtkfont{S} = \omega$.
	Suppose that for $k \in \mathsf{K}$, there exists an Mtk $\Mtk_k : \Y_k \times \F_k \to [0,1]$ such that $\omega_k \Mtk _k = \omega_k$.
	Assume further that there exists $c \in [0,1]$ such that $\Mtk (y,\mathsf{B}) \geq c \Mtk _k(y,\mathsf{B})$ for $k \in \mathsf{K}$, $y \in \Y_k$ and $\mathsf{B} \in \F_k$.
	Then
	\begin{equation} \label{ine:decomposition}
		1 - \|\Mtk \mtkfont{S}^*\|_{\omega}^2 \geq c^2 \left(1 - \sup_{k \in \mathsf{K}} \|\Mtk _k\|_{\omega_k}^2 \right) \mbox{Gap}_{\bar{\omega}} (\overline{\mtkfont{S}^*\mtkfont{S}} ).
	\end{equation}
	In particular, if furthermore $\mtkfont{S} = \Mtk$, then
	\[
	1 - \|\Mtk\|_{\omega}^4 \geq c^2 \left(1 - \sup_{k \in \mathsf{K}} \|\Mtk _k\|_{\omega_k}^2 \right) \mbox{Gap}_{\bar{\omega}} ( \overline{\mtkfont{\Mtk}^*\mtkfont{\Mtk}} ).
	\]
\end{lemma}

{
\begin{remark} \label{rem:madras}
	Lemma \ref{lem:decomposition} extends Theorem A.1 of \cite{madras2002markov}, originally formulated by \cite{caracciolo1992two}.
	From \pcite{caracciolo1992two} result, it can be deduced that, if, in addition to the assumptions in Lemma \ref{lem:decomposition}, $\Mtk$ and $\mtkfont{S}$ are reversible with respect to~$\omega$, $\Mtk_k$ is reversible with respect to $\omega_k$ for $k \in \mathsf{K}$, and the operator $\mtkfont{S}$ is positive semi-definite, then
	\[
	\mbox{Gap}_{\omega}(\mtkfont{S}^{1/2} \Mtk \mtkfont{S}^{1/2}) \geq c \inf_{k \in \mathsf{K}} \mbox{Gap}_{\omega_k}(\Mtk_k) \mbox{Gap}_{\bar{\omega}} (\bar{\mtkfont{S}}).
	\]
	In particular, if furthermore $\mtkfont{S} = \Mtk$, then
	\[
	1 - \|\Mtk\|_{\omega}^2 = \mbox{Gap}_{\omega}(\Mtk^2) \geq c \inf_{k \in \mathsf{K}} \mbox{Gap}_{\omega_k}(\Mtk_k) \mbox{Gap}_{\bar{\omega}} ( \bar{\Mtk} ).
	\]
\end{remark}

\begin{remark}
	The proof of Lemma \ref{lem:decomposition} is given in Section~\ref{app:decomposition} of \ref{s1}.
	It adopts the general idea of the proof of Theorem A.1 of \cite{madras2002markov}, with alterations made to tackle non-reversibility.
	Notably, \cite{madras2002markov} studied how a reversible $\Mtk$ acts on functions of the form $\mtkfont{S}^{1/2} f$ for $f \in L_0^2(\omega)$ by decomposing the Dirichlet form $\|\mtkfont{S}^{1/2}f\|_{\omega}^2 - \langle \mtkfont{S}^{1/2}f, \Mtk \mtkfont{S}^{1/2}f \rangle_{\omega}$ into local and global components.
	Here, we study how a possibly non-reversible $\Mtk$ acts on functions of the form $\mtkfont{S}^* f$ by decomposing $\|\mtkfont{S}^* f \|_{\omega}^2 - \|\Mtk \mtkfont{S}^* f\|_{\omega}^2$.
\end{remark}

\begin{remark} \label{rem:lower}
	Using standard techniques, it is straightforward to derive the following bound in the opposite direction of those given in Lemma \ref{lem:decomposition} and \cite{madras2002markov}:
	$$
	1 - \|\Mtk\|_{\omega}^2 \leq \mbox{Gap}_{\bar{\omega}}(\overline{\Mtk^* \Mtk}) \leq 1 - \|\bar{\mtkfont{T}}\|_{\bar{\omega}}^2.
	$$
	We provide a brief derivation at the end of Section \ref{app:decomposition} in \ref{s1}.
\end{remark}
}

\subsection{Geometric convergence of the trans-dimensional chain} \label{ssec:geometric}

Lemma \ref{lem:decomposition} can be used to construct an upper bound on the norm of $P^t$ for some $t \geq 1$, where~$P$ is the Mtk of the trans-dimensional chain defined in the Introduction.

{
Recall the definitions of $\bar{\mtkfont{S}}$ and $\overline{\mtkfont{S}^*\mtkfont{S}}$, and consider letting $(\Y,\F,\omega) = (\X, \SF, \Pi)$ and $\mtkfont{S} = P$.
Then $\bar{P}$ is defined as in \ref{H2}, and
\[
\begin{aligned}
	\overline{P^*P}(k,\{k'\}) &= \frac{1}{\Psi_k(\Z_k)} \sum_{k'' \in \mathsf{K}} \int_{\Z_{k''}} \Psi_{k''}(\df z) P((k'',z), \{k\} \times \Z_k ) P((k'',z), \{k'\} \times \Z_{k'}).
\end{aligned}
\]
These two Mtks describe the between-model movements of the trans-dimensional chain.
$\bar{P}(k,\{k'\})$ can be understood as the average probability of moving from model~$k$ to model~$k'$.
$\overline{P^*P}(k,\{k'\})$ is similar, but with $P$ replaced by $P^*P$. 
Indeed, under mild conditions, $P^*$ and thus $P^*P$ can be seen as Mtks that leave $\Pi$ invariant \citep{paulin2015concentration,choi2020metropolis}, and one can show that
$
\overline{P^*P}(k,\{k'\}) = \int_{\Z_k} \Phi_k(\df z) P^*P((k,z), \{k'\} \times \Z_{k'}) .
$
$\overline{P^*P}$ is called the ``multiplicative reversibilization" of $P$. 
Multiplicative reversibilizations are commonly investigated for non-reversible chains since at least \cite{fill1991eigenvalue}.
If $P$ defines a self adjoint (resp. positive semi-positive) operator on $L_0^2(\Pi)$, then $\bar{P}$ defines a self adjoint (resp. positive semi-positive) operator on $L_0^2(\bar{\Pi})$, where $\bar{\Pi}(\{k\}) = \Psi_k(\Z_k)/\sum_{k' \in \mathsf{K}} \Psi_{k'}(\Z_{k'})$ for $k \in \mathsf{K}$.
On the other hand, $\overline{P^*P}$ always defines a positive semi-definite operator on $L_0^2(\bar{\Pi})$.
}

We now provide a quantitative bound concerning the convergence rate of the trans-dimensional chain.


\begin{theorem} \label{thm:quantitative}
	Just for this theorem, allow $|\mathsf{K}|$ to be countably infinite.
	Suppose that, for each $k \in \mathsf{K}$, there exists an Mtk $P_k: \Z_k \times \A_k \to [0,1]$ such that $\Phi_k P_k = \Phi_k$.
	Suppose further that, for $k \in \mathsf{K}$, there exists $c_k > 0$ such that $P((k,z), \{k\} \times \mathsf{A}) \geq c_k P_k(z, \mathsf{A})$ for $z \in \Z_k$ and $\mathsf{A} \in \A_k$.
	Then, for any positive integer~$t$,
	\begin{equation} \label{ine:quantitative-1}
		1 - \|P^t\|_{\Pi}^4 \geq \left( \inf_{k \in \mathsf{K}} c_k^t \right)^2 \left( 1 - \sup_{k \in \mathsf{K}} \|P_k^t \|_{\Phi_k}^2 \right) \mbox{Gap}_{\bar{\Pi}}(\overline{P^*P}) .
	\end{equation}
	{If, furthermore, $P$ defines a positive semi-definite operator on $L_0^2(\Pi)$ and $P_k$ is reversible with respect to $\Phi_k$ for $k \in \mathsf{K}$, then there is the simpler bound
	\begin{equation} \label{ine:quantitative-2}
		1 - \|P\|_{\Pi}^2 \geq \left( \inf_{k \in \mathsf{K}} c_k \right) \left[ \inf_{k \in \mathsf{K}} \mbox{Gap}_{\Phi_k}(P_k) \right] \mbox{Gap}_{\bar{\Pi}}(\bar{P}).
	\end{equation}}
\end{theorem}

\begin{proof}
	We will establish \eqref{ine:quantitative-1} using Lemma \ref{lem:decomposition}; \eqref{ine:quantitative-2} can be established in a similar fashion using the original Markov chain decomposition result in Remark \ref{rem:madras}.
	
	Fix a positive integer~$t$.
	In Lemma~\ref{lem:decomposition}, take $(\Y,\F,\omega) = (\X,\SF,\Pi)$, $\Mtk = P^t$ and $\mtkfont{S} = P$.
	For $k \in \mathsf{K}$, let $\Y_k = \{k\} \times \Z_k$.
	Then $\F_k$ consists of sets of the form $\{k\} \times \mathsf{A}$, where $\mathsf{A} \in \A_k$, and $\omega_k(\{k\} \times \mathsf{A}) = \Phi_k(\mathsf{A})$ for $\mathsf{A} \in \A_k$.
	For $k \in \mathsf{K}$, $z \in \Z_k$, and $\mathsf{A} \in \A_k$, let
	$
	\Mtk_k((k,z), \{k\} \times \mathsf{A}) = P_k^t(z, \mathsf{A}).
	$
	Since $\Phi_k P_k = \Phi_k$, it holds that $\omega_k \Mtk_k = \omega_k$.
	Since $P((k,z), \{k\} \times \mathsf{A}) \geq c_k P_k(z,\mathsf{A})$ for $z \in \Z_k$ and $\mathsf{A} \in \mathcal{A}_k$, it holds that, for $(k,z) \in \Y_k$ and $\{k\} \times \mathsf{A} \in \F_k$,
	\[
	\Mtk((k,z), \{k\} \times \mathsf{A}) \geq c_k^t P_k^t(z, \mathsf{A}) \geq c \Mtk_k((k,z), \{k\} \times \mathsf{A}),
	\]
	where $c = \inf_{k \in \mathsf{K}} c_k^t$.
	Thus, the assumptions of Lemma~\ref{lem:decomposition} are satisfied.
	
	The next step is identifying the objects in \eqref{ine:decomposition}.
	Obviously, $\|\Mtk \mtkfont{S}^*\|_{\omega} = \|P^t P^*\|_{\Pi}$.
	Standard arguments show that $\|\Mtk_k\|_{\omega_k} = \|P_k^t\|_{\Phi_k}$.
	The distribution $\bar{\omega}$ corresponds to $\bar{\Pi}$, while the Mtk $\overline{\mtkfont{S}^* \mtkfont{S}}$ is $\overline{P^*P}$.
	Then, by Lemma~\ref{lem:decomposition},
	\begin{equation} \nonumber
		1 - \|P^t P^*\|_{\Pi}^2 \geq c^2 \left( 1 - \sup_{k \in \mathsf{K}} \|P_k^t \|_{\Phi_k}^2 \right) \mbox{Gap}_{\bar{\Pi}}(\overline{P^*P}). 
	\end{equation}
	Finally, note that
	\[
	\|P^t\|_{\Pi}^2 = \|P^t P^{*t}\|_{\Pi} \leq \|P^t P^*\|_{\Pi} \|P^{* t-1}\|_{\Pi} = \|P^t P^*\|_{\Pi} \|P^{t-1}\|_{\Pi}  \leq \|P^t P^*\|_{\Pi}.
	\]
	The desired result then follows.
\end{proof}

{
Theorem \ref{thm:quantitative} connects the convergence properties of the trans-dimensional chain, quantified by $\|P\|_{\Pi}$, to the convergence properties of the within- and between-model movements, quantified by the $\|P_k\|_{\Phi_k}$'s and $\mbox{Gap}_{\bar{\Pi}}(\overline{P^*P})$ (or $\mbox{Gap}_{\bar{\Pi}}(P)$), respectively.
In Section \ref{ssec:toy}, we use a toy example to test the sharpness of the bounds in Theorem \ref{thm:quantitative}, and investigate how the within- and between-model components may affect $\|P\|_{\Pi}$ and its bound.

\begin{remark} \label{rem:lower-2}
	By Remark~\ref{rem:lower}, we also have $1 - \|P\|_{\Pi}^2 \leq \mbox{Gap}_{\bar{\Pi}}(\overline{P^*P})$, and if $P$ is reversible, $1 - \|P\|_{\Pi} \leq \mbox{Gap}_{\bar{\Pi}}(\bar{P})$.
	Thus, $\|P\|_{\Pi}$ is controlled by $\mbox{Gap}_{\bar{\Pi}}(\overline{P^*P})$ from above and below.
	In particular, combining \eqref{ine:quantitative-1} with the above yields
	\[
	\frac{1}{4} \left( \inf_{k \in \mathsf{K}} c_k \right)^2 \left( 1 - \sup_{k \in \mathsf{K}} \|P_k \|_{\Phi_k}^2 \right) \leq \frac{1 - \|P\|_{\Pi}}{\mbox{Gap}_{\bar{\Pi}}(\overline{P^*P})} \leq 1.
	\]
	Similarly, if $P$ is positive semi-definite and $P_k$ is reversible for each~$k$,
	\[
	\frac{1}{2} \left( \inf_{k \in \mathsf{K}} c_k \right) \left[ \inf_{k \in \mathsf{K}} \mbox{Gap}_{\Phi_k}(P_k) \right] \leq \frac{1 - \|P\|_{\Pi}}{\mbox{Gap}_{\bar{\Pi}}(\bar{P}) } \leq 1.
	\]
\end{remark}

}

Quantities such as $\|\overline{P^*P}\|_{\Pi}$ and $\|P_k\|_{\Phi_k}$ may be difficult to compute in practice.
However, when $|\mathsf{K}| < \infty$, Theorem~\ref{thm:quantitative} immediately yields Theorem~\ref{thm:main}, which is stated again below:

\noindent{\bf Theorem \ref{thm:main}.} 
{\it
	Assume that \ref{H1} and \ref{H2} hold.
	Then $\|P^{t_0}\|_{\Pi} < 1$, and~$P$ is $L^2(\Pi)$ geometrically convergent.
}

\begin{proof}
	By Theorem~\ref{thm:quantitative}, it suffices to show that 
	\[
	\left( \min_{k \in \mathsf{K}} c_k^{t_0} \right)^2 \left( 1 - \max_{k \in \mathsf{K}} \|P_k^{t_0} \|_{\Phi_k}^2 \right) \mbox{Gap}_{\bar{\Pi}}(\overline{P^*P}) > 0.
	\]
	
	By (iii) in \ref{H1}, $\min_{k \in \mathsf{K}} c_k^{t_0} > 0$.
	By (ii) in \ref{H1}, $\max_{k \in \mathsf{K}} \|P_k^{t_0} \|_{\Phi_k}^2 < 1$.
	
	{
	It remains to show that $\mbox{Gap}_{\bar{\Pi}}(\overline{P^*P}) > 0$.
	Assume the opposite, i.e., $\mbox{Gap}_{\bar{\Pi}}(\overline{P^*P}) = 0$.
	Because $\mathsf{K}$ is finite and $\overline{P^*P}$ is reversible, this implies that the largest eigenvalue of $\overline{P^*P}$ is 1.
	It then follows that $\overline{P^*P}$ is reducible \citep[][Theorem 3.11]{hairer2006ergodic}.
	By (iii) in \ref{H1}, for $k, k' \in \mathsf{K}$,
	\[
	\begin{aligned}
		\overline{P^*P}(k, \{k'\}) &\geq \int_{\Z_k} \Phi_k(\df z) P((k,z), \{k\} \times \Z_k ) P((k,z), \{k'\} \times \Z_{k'}) \geq c_k \bar{P}(k,\{k'\}).
	\end{aligned}
	\]
	So $\bar{P}$ must be reducible as well.
	But this contradicts with \ref{H2}.
	Hence, $\mbox{Gap}_{\bar{\Pi}}(\overline{P^*P}) > 0$.
	}

\end{proof}

Theorem~\ref{thm:main} will be used to establish geometric convergence for two practical examples in Section~\ref{sec:examples} and another one in \ref{s2}.

\section{Examples} \label{sec:examples}

This section contains a toy example and two practical problems concerning variable selection and mixture models respectively.
For a third practical example concerning autoregression, see \ref{s2}.

{
\subsection{A toy chain} \label{ssec:toy}
}

We first use a toy algorithm to test the sharpness of the quantitative bounds in Theorem~\ref{thm:quantitative}.


Let $k_{\scriptsize\mbox{max}}$ and $n$ be positive integers.
Let $\mathsf{K} = \{1,\dots,k_{\scriptsize\mbox{max}}\}$.
Consider a simple scenario where all the $\Psi_k$'s are the same.
To be specific, for $k = 1,\dots,k_{\scriptsize\mbox{max}}$, let $\Z_k = \{1,\dots,n\}$, and let $\Psi_k$ be the counting measure on $\Z_k$.
Then $\Pi$ is the uniform distribution on $\X = \bigcup_{k \in \mathsf{K}} \{k\} \times \Z_k$.

We consider a type of MCMC algorithm targeting $\Pi$.
Given the current state $(k,z) \in \X$, the algorithm either makes a local or a global move, each with probability $1/2$.
A local move sends the underlying Markov chain to $(k,z')$ where $z' \in \Z_k$, and a global move sends the chain to $(k',z)$ where $k' \in \mathsf{K}$.
We consider three types of local moves and two types of global moves.
The three types of local moves are ``fast," ``slow," and ``varied."
In a fast local move, the chain randomly and uniformly selects $z' \in \Z_k$ to move to.
In a slow local move, the chain can either stay in place with probability $1/2$, or move to one of the within-model neighbors.
Within a model~$k$, two states $z$ and $z'$ are neighbors if they differ by~1, or if one of them is~1 and the other is~$n$.
In a varied local move, the movement of the chain varies with the current model~$k$.
To be precise, the chain stays in place with probability $1-1/k$, and move randomly and uniformly across $\Z_k$ with probability $1/k$.
The two types of global moves are ``fast" and ``slow."
In a fast global move, the chain randomly and uniformly selects a new model $k' \in \mathsf{K}$ to move to.
In a slow global move, the chain can stay in place with probability $1/2$, or move to one of the neighboring models.
Two models $k$ and $k'$ are neighbors if they differ by~1, or if one of them is~1 and the other is $k_{\scriptsize\mbox{max}}$.
Each combination of local and global move types gives rise to a concrete algorithm, and we may define six algorithms in this manner.

For a given algorithm, let $P$ be the Mtk of the underlying chain, and let $P_k$ be the Mtk associated with the local movement within model~$k$.
One can check that, for each of the six algorithms, $P$ is positive semi-definite, and so are the $P_k$'s.
The local and global behavior of~$P$ is summarized as follows.
\begin{itemize}
	\item When the local move type is fast, $\|P_k\|_{\Phi_k} = 0$.
	When the local move type is slow, $\|P_k\|_{\Phi_k}$ is a function of~$n$, and it goes to~1 as $n \to \infty$.
	When the local move is varied, $\|P_k\|_{\Phi_k} = 1-1/k$.
	\item When the global move type is fast, $\mbox{Gap}_{\bar{\Pi}}(\overline{P^*P}) = 3/4$ and $\mbox{Gap}_{\bar{\Pi}}(\bar{P}) = 1/2$.
	When the local move type is slow, $\mbox{Gap}_{\bar{\Pi}}(\overline{P^*P})$ is a function of $k_{\scriptsize\mbox{max}}$, and it goes to 0 as $k_{\scriptsize\mbox{max}} \to \infty$; the same goes for $\mbox{Gap}_{\bar{\Pi}}(\bar{P})$.
\end{itemize}

\begin{figure}
	\begin{center}
		\includegraphics[width=\textwidth]{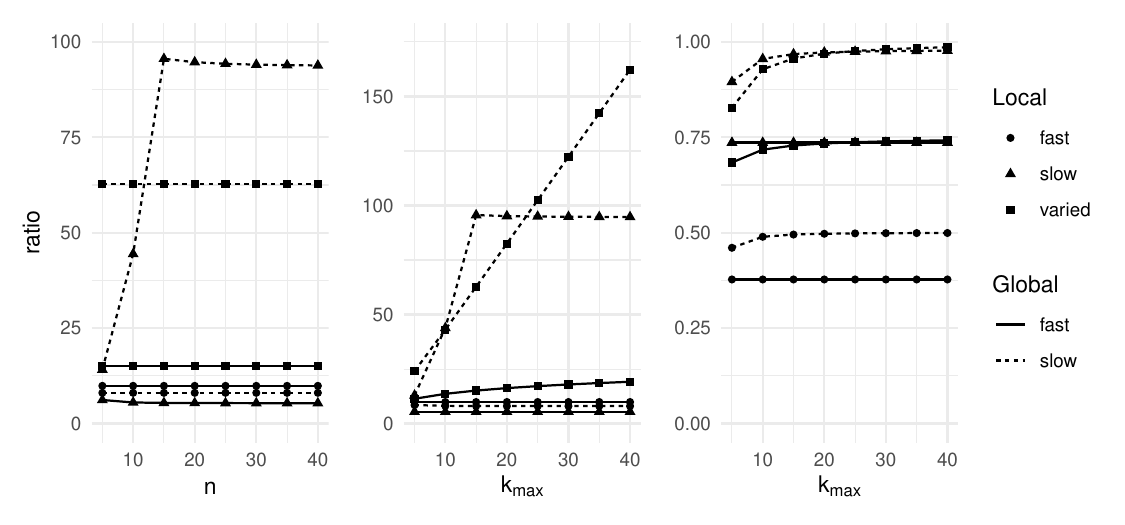}
		\caption{Performance of the quantitative bounds for the toy chain.
		Left: $(1-\|P\|_{\Pi})/(1-\|P\|_{\Pi}^{\dagger})$ for various values of~$n$ when $k_{\scriptsize\mbox{max}} = 15$.
		Middle: $(1-\|P\|_{\Pi})/(1-\|P\|_{\Pi}^{\dagger})$ for various values of $k_{\scriptsize\mbox{max}}$ when $n = 15$.
		Right: $(1-\|P\|_{\Pi}^{\dagger})/(1-\|P\|_{\Pi}^{\ddagger})$ for various values of $k_{\scriptsize\mbox{max}}$ when $n = 15$.
		} \label{fig:numeric}
	\end{center}
\end{figure}

Using Theorem \ref{thm:quantitative}, we may construct upper bounds on $\|P\|_{\Pi}$.
The bound derived from \eqref{ine:quantitative-1} with $c_k = 1/2$ and $t=1$ will be denoted by $\|P\|_{\Pi}^{\dagger}$, while the bound derived from \eqref{ine:quantitative-2}, which exploits reversibility, is denoted by $\|P\|_{\Pi}^{\ddagger}$.

Since $\X$ is finite, the true value of $\|P\|_{\Pi}$ can be computed.
We test the performance of the bounds through numerical simulation, and the results are given in Figure \ref{fig:numeric}.
Table \ref{tab:bound} loosely summarizes how the local and global move types affect the sharpness of $\|P\|_{\Pi}^{\dagger}$.
The bound $\|P\|_{\Pi}^{\ddagger}$ is sharper than $\|P\|_{\Pi}^{\dagger}$, but the two bounds are comparable.

\begin{table}
	\caption{Performance of the bound \eqref{ine:quantitative-1} for the toy chain} \label{tab:bound}
	\begin{center}
		\begin{tabular}{l|lll}
			\hline
			& Locally fast & Locally slow & Locally varied \\
			\hline
			Globally fast & well behaved & well behaved & deteriorates slowly as $k_{\scriptsize\mbox{max}} \to \infty$ \\
			Globally slow & well behaved & situation dependent & deteriorates as $k_{\scriptsize\mbox{max}} \to \infty$ \\
			\hline
		\end{tabular}
	\end{center}
\end{table}

\subsection{Variable selection in Bayesian probit regression}

\subsubsection{The model}


For $i = 1,\dots,n$, let $x_i = (x_{i,1}, \dots, x_{i,r})^{\top} \in \mathbb{R}^r$ be a known vector of predictors.
Let $Y_1, \dots, Y_n$ be independent binary responses, where  $Y_i$ follows a Bernoulli distribution with success probability $F(A+x_i^{\top} B)$, with $F(\cdot)$ being the cumulative distribution function of the standard normal distribution.
The scalar $A \in \mathbb{R}$ is an unknown intercept, while the vector $B= (B_1, \dots, B_r)^{\top} \in \mathbb{R}^r$ is an unknown regression coefficient.

To perform Bayesian variable selection, put a spike and slab prior on~$B$.
To be specific, let $K = (K_1, \dots, K_r)^{\top} \in \{0,1\}^r$.
Let $\mathsf{J}_K = \{j \in \{1,\dots,r\}: \, K_j = 1\}$.
Place a prior distribution on~$K$ that has probability mass function proportional to $p^{|\mathsf{J}_k|}$, where $p \in (0,1)$ is a hyperparameter.
Assume that given~$K$, the $B_j$'s are independent.
If $K_j = 0$, $B_j$ is set to be zero; otherwise, $B_j$ follows the  $\mbox{N}(0,\sigma^2)$ distribution, where $\sigma$ is a positive hyperparameter.
Thus, $K$ indicates the collection of relevant predictors.
The intercept $A$ is independent of $(B,K)$ and follows the $\mbox{N}(0,\sigma^2)$ distribution.
Let $Z = (A, B^{\top})^{\top}$.
Having observed $Y = (Y_1, \dots, Y_n)^{\top}$, the goal is to make inference about $(K, Z)$.

The parameter space, i.e., the range of $(K,Z)$, is $\X = \bigcup_{k \in \mathsf{K}} \{k\} \times \Z_k$, where $\mathsf{K} = \{0,1\}^r$, and,
for $k = (k_1, \dots,k_r) \in \mathsf{K}$,
$\Z_k$ is the set of $(\alpha, \beta_1, \dots, \beta_r) \in \mathbb{R}^{r+1}$ such that $\beta_j = 0$ whenever $k_j = 0$.
The posterior distribution~$\Pi$ of $(K,Z)$ given $Y = y = (y_1, \dots, y_n)$ is of the form~\eqref{eq:Pi}.
Evaluated at $z = (\alpha,\beta^{\top})^{\top} \in \Z_k$, the density function of $\Psi_k$ is
\[
\begin{aligned}
	\pi(k, z \mid y) =& \frac{1}{\sqrt{2\bm{\pi}} \sigma} \left( \frac{p}{\sqrt{2 \bm{\pi}} \sigma} \right)^{|\mathsf{J}_k|} \exp \left( - \frac{ \alpha^2 + \sum_{j \in \mathsf{J}_k} \beta_j^2 }{2 \sigma^2} \right) \\
	&\prod_{i=1}^n F \left( \alpha + \sum_{j \in \mathsf{J}_k} x_{i,j} \beta_j \right)^{y_i} \left[ 1 - F \left( \alpha + \sum_{j \in \mathsf{J}_k} x_{i,j} \beta_j \right)  \right]^{1-y_i}.
\end{aligned}
\]

\subsubsection{A reversible jump MCMC algorithm}

For a fixed model $k=(k_1,\dots,k_r)$, one can use a well-known data augmentation algorithm (a type of reversible MCMC algorithm) devised by \cite{albert1993bayesian} to sample from $\Phi_k$, the normalization of $\Psi_k$.
	The algorithm is now briefly described.
	Suppose that $\mathsf{J}_k = \{j_1, \dots, j_d\}$, where $d = |\mathsf{J}_k|$ and $j_1 < \cdots < j_d$.
	Let $M(k)$ denote the $n \times (d+1)$ matrix whose $i$th row is $(1, x_{i,j_1}, \dots, x_{i,j_d} )^{\top}$.
	Given the current state $z = (\alpha, \beta_1, \dots, \beta_r)^{\top} \in \Z_k$, the next state $z'= (\alpha', \beta'_1, \dots, \beta'_r)^{\top}$ is drawn through the following procedure:
	Independently for $i = 1,\dots,n$, draw $u_i$ from the $\mbox{N}( \alpha + \sum_{j=1}^r x_{i,j} \beta_j , 1^2 )$ distribution truncated to $(0,\infty)$ if $y_i = 1$ and to $(-\infty,0)$ otherwise.
	Let $u = (u_1, \dots, u_n)^{\top}$.
	Draw $(\alpha', \beta'_{j_1}, \dots, \beta'_{j_d})^{\top}$ from the normal distribution
	\[
	\mbox{N}_{d+1} \left( \left[M(k)^{\top}M(k) + \sigma^{-2} I_{d+1} \right]^{-1} M(k)^{\top} u, \; \left[M(k)^{\top} M(k) + \sigma^{-2} I_{d+1} \right]^{-1}  \right),
	\]
	and set the remaining elements of $z'$ to zero.

The reversible jump MCMC algorithm considered herein is a combination of the data augmentation algorithm and a standard reversible jump scheme.
Given the current state $(k,z)$, the algorithm randomly performs a U (update), B (birth), or D (death) move.
It is assumed that the probability of choosing a move depends on $(k,z)$ only through~$k$, and the probabilities are denoted by $q_{\scriptsize\mbox{U}}(k)$, $q_{\scriptsize\mbox{B}}(k)$, and $q_{\scriptsize\mbox{D}}(k)$, respectively.

\begin{itemize}
	\item {\it U move}: 
	Draw $z'$ using one iteration of Albert and Chib's data augmentation algorithm.
	Set the new state to $(k,z')$.
	
	\item {\it B move}:
	Randomly and uniformly choose an index $j$ from $\mathsf{J}_k^c$, where the complement is taken with respect to the set $\{1,\dots,r\}$.
	Change the $j$th element of~$k$ to~1 and call the resultant binary vector $k'$.
	Draw $b_*$ from some distribution on~$\mathbb{R}$ associated with a density function $g(\cdot \mid k, k', z, y)$.
	Replace the $(j+1)$th element of $z = (\alpha, \beta_1, \dots, \beta_r)^{\top}$ by $b_*$, and call the resultant vector $z'$.
	With probability
	\[
	\min \left\{ 1, \frac{ \pi(k', z' \mid y) \, q_{\scriptsize\mbox{D}}(k') \,  (|\mathsf{J}_k|+1)^{-1} }{ \pi(k,z \mid y) \, q_{\scriptsize\mbox{B}}(k) \, (r - |\mathsf{J}_k|)^{-1} g(b_* \mid k, k', z, y)  }  \right\},
	\]
	set the next state to $(k',z')$;
	otherwise, keep the old state.
	This move type is available only when $|\mathsf{J}_k| < r$.
	
	\item {\it D move}:
	Randomly and uniformly choose an index $j$ from $\mathsf{J}_k$.
	Change the $j$th element of~$k$ to 0 and call the resultant binary vector $k'$.
	Let $\beta_j$ be the $(j+1)$th element of $z$.
	Let $z'$ be the vector obtained by changing the $(j+1)$th element of $z$ to 0.
	With probability
	\[
	\min \left\{ 1, \frac{\pi(k',z' \mid y) \, q_{\scriptsize\mbox{B}}(k') \, (r - |\mathsf{J}_k| + 1)^{-1} g(\beta_j \mid k', k, z', y)}{\pi(k, z \mid y) \, q_{\scriptsize\mbox{D}}(k) \, |\mathsf{J}_k|^{-1}} \right\},
	\]
	set the next state to $(k',z')$;
	otherwise, keep the old state.
	This move type is available only when $|\mathsf{J}_k| > 0$.
\end{itemize}

The resultant trans-dimensional chain is reversible with respect to~$\Pi$.

\subsubsection{Convergence analysis}

For $k \in \mathsf{K}$, let $P_k$ be the Mtk of the data augmentation chain targeting $\Phi_k$.
\cite{chakraborty2016convergence} proved the following result regarding $P_k$ using the drift and minorization technique.
See also \cite{roy2007convergence}.

\begin{lemma} \label{lem:acgeo}
	For $k \in \mathsf{K}$, $P_k$ is $\Phi_k$-a.e. geometrically ergodic.
\end{lemma}

Given~$k$, the data augmentation Mtk $P_k$ is reversible with respect to $\Phi_k$.
One can then establish geometric convergence for the reversible jump chain.

\begin{proposition} \label{pro:acgeo}
	Suppose that $q_{\scriptsize\mbox{U}}(k) > 0$ for $k \in \mathsf{K}$, $q_{\scriptsize\mbox{B}}(k) > 0$ when $|\mathsf{J}_k| < r$, $q_{\scriptsize\mbox{D}}(k) > 0$ when $|\mathsf{J}_k| > 0$.
	Then the reversible jump chain is $L^2(\Pi)$-geometrically convergent and $\Pi$-a.e. geometrically ergodic.
\end{proposition}

\begin{proof}
	Apply Theorem~\ref{thm:main}.
	Let~$P$ be the Mtk of the reversible jump chain.
	By Lemma~\ref{lem:acgeo} and (i) of Lemma~\ref{lem:Pk-convergence}, $\|P_k\|_{\Phi_k} < 1$ for $k \in \mathsf{K}$.
	It follows that, for $k \in \mathsf{K}$, (i) and (ii) of \ref{H1} hold with $t_0 = 1$.
	Evidently, (iii) of \ref{H1} also holds with $c_k = q_{\scriptsize\mbox{U}}(k)$.
	
	To verify \ref{H2}, note that for $k,k' \in \mathsf{K}$, $\bar{P}(k,\{k'\}) > 0$ whenever $k$ and $k'$ differ by at most 1 element.
	Then it is clear that $\bar{P}$ is irreducible.
	(Alternatively, note that $P$ is $\Pi$-irreducible.)

	The chain is thus $L^2(\Pi)$ geometrically convergent.
	By Theorem 1 of \cite{roberts2001geometric}, it is $\Pi$-a.e. geometrically ergodic.
\end{proof}

\subsubsection{Application to a data set} \label{sssec:binary-example}

Geometric ergodicity allows one to estimate the importance of features in a variable selection problem with confidence.
The reversible jump algorithm is applied to the \verb|Spambase| data set \citep{misc_spambase_94}.
This data set contains $n = 4601$ emails. 
The response $Y_i$ indicates whether the $i$th email is spam.
Each email is associated with $r = 57$ attributes, including the frequency of certain words and the length of sequences of consecutive capital letters.
To perform variable selection, a spike and slab prior with $p = 0.5$ is used.

In the B move of the reversible jump algorithm, $g(\cdot \mid k, k', z, y)$ is chosen to be the density of a normal distribution, whose parameters are selected using ideas from \cite{brooks2003efficient}.
The probabilities of proposing birth and death moves are as follows:
\[
q_{\scriptsize\mbox{B}}(k) = \frac{1}{3} \min \left\{ 1, \frac{p \, (r - |\mathsf{J}_k|)}{|\mathsf{J}_k| + 1} \right\}, \quad q_{\scriptsize\mbox{D}}(k) = \frac{1}{3} \min \left\{ 1, \frac{|\mathsf{J}_k|}{p \, (r - |\mathsf{J}_k| + 1)} \right\}.
\]
By Proposition~\ref{pro:acgeo}, the reversible jump chain is $\Pi$-a.e. geometrically ergodic.

\begin{figure}
	\centering
	\includegraphics[width=0.7\textwidth]{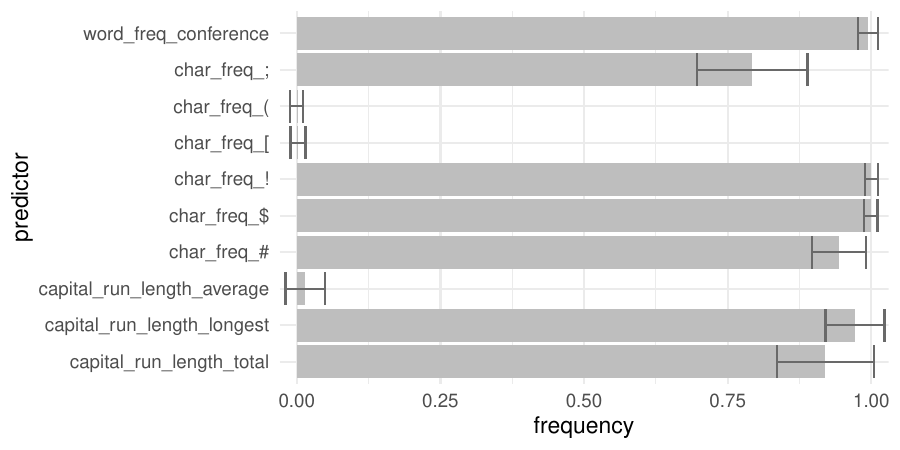}
	\caption{Estimated posterior probabilities of selecting each variable (attribute), with 95\% simultaneous confidence intervals.
		Only the last 10 variables are shown.
	} \label{fig:probit}
\end{figure}

A chain $(K(t), Z(t))_{t=1}^m$ of length $m = 10^5$ is simulated.
The quantities of interest are the posterior probabilities of $K_j = 1$ for $j = 1,\dots,57$, i.e., the posterior probability of any given predictor being present in the regression model.
They are estimated using the sample proportions $m^{-1} \sum_{t=1}^m K_j(t)$.
Under geometric ergodicity, the sample proportions are asymptotically normally distributed, and the asymptotic variances can be consistently estimated using the batch means method \citep{jones2006fixed}.
To avoid underestimation, which is a problem exhibited by batch means estimators when the Monte Carlo sample size is not sufficiently large \citep[][Section 4]{flegal2010batch}, we add $(\log m) \sqrt{1/b_m^2+b_m/m}$ to the estimated asymptotic variances, where $b_m \approx m^{0.6}$ is the batch size, and $1/b_m^2 + b_m/m$ is on the same order as the mean squared error of the batch means estimator \citep[][Section 3]{flegal2010batch}.
This adjustment is further discussed in Section \ref{app:wald} of \ref{s1}.
We construct 95\% simultaneous Wald confidence intervals for the posterior probabilities.
Bonferroni correction is used here, although more sophisticated multivariate methods could be considered; see Section \ref{sec:discussion}.
The confidence intervals for the last 10 variables (attributes) are presented in Figure~\ref{fig:probit}.
This figure shows how important each predictor is according to the MCMC simulation, as well as the errors of their estimated importance.

{\subsection{Gaussian mixture model} \label{ssec:mixture}}

\subsubsection{The model}


Let $Y_1, \dots, Y_n$ be iid random variables drawn from the mixture of $K$ normal distributions.
For $j = 1,\dots,K$, let $W_j$ be the weight associated with the $j$th normal distribution, and let $U_j$ and $T_j$ be, respectively, the mean and variance of that normal distribution.
Equivalently, we may formulate the model as follows.
Let $(Y_1, A_1), \dots, (Y_n, A_n)$ be iid random vectors that take values in $\mathbb{R} \times \{1,\dots,k\}$, where, for each $i$, $Y_i$ given $A_i$ follows the $\mbox{N}(U_{A_i}, T_{A_i})$ distribution, and marginally $P(A_i = j) = W_j$ for $j = 1,\dots,K$.
Suppose that $Y = (Y_1, \dots, Y_n)$ is observable, while $A = (A_1, \dots, A_n)$, $W = (W_1, \dots, W_K)$, $U = (U_1, \dots, U_K)$, $T = (T_1, \dots, T_K)$, and~$K$ are unknown.

To perform Bayesian analysis, we assume that $K$ has a prior probability mass function $k \mapsto f_K(k)$ that is supported on a finite set $\mathsf{K} = \{1, \dots, k_{\scriptsize\max}\}$.
We then put a Dirichlet prior on $W$, inverse gamma priors on $T$, and normal priors on $U$.
To address the label switching problem and enforce identifiability, we shall assume that $U_1 \leq U_2 \leq \cdots \leq U_K$ in the prior distribution.
To be precise, we shall assume that, given $K = k$, the prior density function of $(W,T,U)$ evaluated at $(w,\tau,u) = ((w_1, \dots, w_k), (\tau_1, \dots, \tau_k), (u_1, \dots, u_k))$ has the following form:
\begin{equation} \label{eq:fWTU}
	\begin{aligned}
		f_{W,T,U}(w,\tau, u \mid k) =& \left\{ \prod_{j=1}^k w_j^{\gamma - 1} \frac{b^c}{\Gamma(c)} \tau_j^{-c - 1} e^{-b/\tau_j} \frac{1}{\sqrt{2 \bm{\pi} \tau_0 \tau_j}} \exp\left[ - \frac{(u_j-u_0)^2}{2\tau_0 \tau_j}  \right] \ind_{(0,\infty)}(\tau_j) \right\} \\
		&  \frac{\Gamma(k\gamma)}{\Gamma(\gamma)^k} \ind_{\mathsf{S}_k}(w) \, k! \, \ind_{\mathsf{G}_k}(u).
	\end{aligned}
\end{equation}
In the above display, $\gamma, b, c, \tau_0, u_0$ are positive hyperparameters, $\Gamma(\cdot)$ is the gamma function, $\mathsf{S}_k = \{(w_1, \dots, w_k) \in (0,1)^k: \, \sum_{j=1}^k w_j = 1 \}$, $\mathsf{G}_k = \{(u_1, \dots, u_k) \in \mathbb{R}^k: \, u_1 \leq \cdots \leq u_k\}$.

The un-normalized posterior density of $(K,A,W,T,U)$ is then
\[
\pi(k,\alpha,w,\tau,u \mid y) = \left\{ \prod_{i=1}^n w_{\alpha_i} \frac{1}{\sqrt{2\bm{\pi} \tau_{\alpha_i} } } \exp\left[ - \frac{1}{2 \tau_{\alpha_i}} (y_i - u_{\alpha_i})^2 \right] \right\} f_{W,T,U}(w,\tau,u \mid k) f_K(k),
\]
where $\alpha_i$ denotes the $i$th element of $\alpha$.
The corresponding measure $\Pi$ has the form \eqref{eq:Pi}, with the density of $\Psi_k$ given by $(\alpha,w,\tau,u) \mapsto \pi(k, \alpha,w,\tau,u \mid y)$.
In this context, $\Z_k = \{1,\dots,k\}^n \times \mathsf{S}_k \times (0,\infty)^k \times \mathsf{G}_k$.

\subsubsection{A reversible jump algorithm}

The algorithm we consider is modified from \pcite{richardson1997bayesian} reversible jump algorithm, with a new within-model move type and simplified between-model move types.

We shall first propose an algorithm for sampling from $\Phi_K$ when $K$ is known.
For $k \in \mathsf{K}$ and $(w,\tau,u) \in \mathsf{S}_k \times (0,\infty)^k \times \mathbb{R}^k$, let $\tilde{f}_{W,T,U}(w,\tau,u \mid k)$ be the same as \eqref{eq:fWTU} but without the constraint $\ind_{\mathsf{G}_k}(u)$.
It can be shown that, given $k \in \mathsf{K}$, for $\alpha \in \{1,\dots,k\}^n$,
\[
\begin{aligned}
	&\tilde{\pi}_A(\alpha \mid k, y) \\
	:=& \int_{\mathsf{S}_k \times (0,\infty)^p \times \mathbb{R}^p} \left\{ \prod_{i=1}^n w_{\alpha_i} \frac{1}{\sqrt{2\bm{\pi} \tau_{\alpha_i} } } \exp\left[ - \frac{1}{2 \tau_{\alpha_i}} (y_i - u_{\alpha_i})^2 \right] \right\}  \tilde{f}_{W,T,U}(w,\tau,u) \, \df (w, \tau, \mu) \\
	\propto & \prod_{j=1}^k  \frac{ [\tau_0 \, n_j(\alpha) + 1]^{-1/2} \, \Gamma(n_j(\alpha) + \gamma) \, \Gamma(n_j(\alpha)/2 + c) }{ \left\{ \mathit{ss}_j(\alpha)/2 +u_0^2/(2\tau_0) -[ s_j(\alpha) + u_0/\tau_0]^2/[2 (n_j(\alpha) + 1/\tau_0)] + b \right\}^{n_j(\alpha)/2 + c} },
\end{aligned}
\]
where $n_j(\alpha) = \sum_{i=1}^n \ind_{\{j\}}(\alpha_i)$, $s_j(\alpha) = \sum_{i=1}^n  y_i \ind_{\{j\}}(\alpha_i)$, and $\mathit{ss}_j(\alpha) = \sum_{i=1}^n y_i^2 \ind_{\{j\}}(\alpha_i)$.
Let $z = (\alpha,w,\tau,u) \in \Z_k$ be the current state.
The next state $z' = (\alpha',w',\tau',u')$ is drawn via the following steps:
\begin{enumerate}
	\item Let $\alpha^{(0)} = (\sigma(\alpha_1), \dots, \sigma(\alpha_n))$, where $\sigma$ is a randomly and uniformly selected permutation of $\{1,\dots,k\}$.
	\item 
	For $i$ from 1 to $n$, do the following.
	Randomly and uniformly draw $j_i$ from $\{1,\dots,k\}$.
	Let $\alpha^{(i-1)}_{i \leftarrow j_i} \in \{1,\dots,k\}^n$ be the vector obtained from $\alpha^{(i-1)}$ by changing its~$i$th element to~$j_i$.
	With probability
	\[
	\min \left\{ 1, \frac{\tilde{\pi}_A(\alpha^{(i-1)}_{i \leftarrow j_i} \mid k, y)}{\tilde{\pi}_A(\alpha^{(i-1)} \mid k, y)} \right\},
	\]
	let $\alpha^{(i)} = \alpha^{(i-1)}_{i \leftarrow j_i}$; otherwise, let $\alpha^{(i)} = \alpha^{(i-1)}$.
	Denote $\alpha^{(n)}$ by $\alpha''$.
	
	\item 
	Draw $w'' = (w''_1, \dots, w''_k)$ from the Dirichlet distribution with concentration parameter $(n_1(\alpha'') + \gamma, \dots, n_k(\alpha'') + \gamma)$.
	
	\item 
	For $j = 1,\dots, k$, independently draw $\tau''_j$ from the inverse gamma distribution with shape parameter $c + n_j(\alpha'')/2$ and scale parameter
	\[
	\frac{\mathit{ss}_j(\alpha'')}{2} + \frac{u_0^2}{2\tau_0} - \frac{[s_j(\alpha'')+\mu_0/\tau_0]^2}{2[ n_j(\alpha'') + 1/\tau_0]} + b.
	\]
	
	\item For $j = 1,\dots,k$, independently draw $u_1'', \dots, u_k''$ from the normal distribution
	\[
	\mbox{N} \left( \frac{s_j(\alpha'') + u_0/\tau_0}{n_j(\alpha'') + 1/\tau_0}, \, \frac{\tau_j''}{n_j(\alpha'') + 1/\tau_0} \right).
	\]
	
	\item 
	Order $u_1'', \dots, u_j''$ so that we find distinct indices $j_1, \dots, j_k$ such that $u_{j_1}'' \leq \cdots \leq u_{j_k}''$.
	For $i = 1,\dots,n$, find the index $\ell_i \in \{1,\dots,k\}$ such that $\alpha_i'' = j_{\ell_i}$, and let $\alpha_i' = \ell_i$.
	For $\ell = 1,\dots,k$, let $w_{\ell}' = w_{j_{\ell}}''$, $\tau_{\ell}' = \tau_{j_{\ell}}''$, and $u_{\ell}' = u_{j_{\ell}}''$.
\end{enumerate}

We shall call the sampler the Metropolis re-ordering algorithm.
The following lemma is proved in Section~\ref{app:reordering-invariant} of \ref{s1}.
\begin{lemma} \label{lem:reordering}
	The underlying Markov chain of the Metropolis re-ordering algorithm leaves $\Phi_k$ invariant for $k \in \mathsf{K}$.
\end{lemma}
A nice property of the Metropolis re-ordering algorithm is that starting from any allocation $\alpha \in \{1,\dots,k\}^n$, it is possible for the chain to move to any allocation $\alpha' \in \{1,\dots,k\}^n$ in a single iteration.

Consider now a reversible jump algorithm for sampling from~$\Pi$, which is part of an algorithm constructed by \cite{richardson1997bayesian}.
When the current state is $(k,z) = (k,\alpha,w,\tau,u)$, the algorithm randomly performs a U, B, or D move.
It is assumed that the probabilities of choosing these moves depend on the current state only through~$k$, and are, respectively, $q_{\scriptsize\mbox{U}}(k)$, $q_{\scriptsize\mbox{B}}(k)$, $q_{\scriptsize\mbox{D}}(k)$.
The three move types are defined as follows:

\begin{itemize}
	\item {\it U move}:
	Draw $z'$ using one iteration of the Metropolis re-ordering algorithm.
	Set the new state to $(k,z')$.
	
	\item {\it B move}:
	Draw $(w_*, \tau_*, u_*)$ from some distribution on $(0,1) \times (0,\infty) \times \mathbb{R}$ with density function $g(\cdot \mid k, \alpha, w, \tau, u)$.
	Find $\ell \in \{1,\dots,k+1\}$ such that $u_j \leq u_*$ whenever $j < \ell$, and $u_j \geq u_*$ whenever $j \geq \ell$ --- that is, $u_*$ is the $\ell$th smallest number in the set $\{u_1, \dots, u_k, u_*\}$.
	Let $\alpha' = (\alpha_1', \dots, \alpha_n')$ be such that, for $i \in \{1,\dots,n\}$, $\alpha_i' = \alpha_i$ if $\alpha_i < \ell$ and $\alpha_i' = \alpha_i + 1$ if $\alpha_i \geq \ell$.
	Let
	\[
	\begin{aligned}
		&w' = ((1-w_*)w_1, \dots, (1-w_*) w_{\ell-1}, w_*, (1-w_*) w_{\ell}, \dots, (1-w_*) w_k ), \\
		&\tau' = (\tau_1, \dots, \tau_{\ell-1}, \tau_*, \tau_{\ell}, \dots, \tau_k), \quad u' = (u_1, \dots, u_{\ell-1}, u_*, u_{\ell}, \dots, u_k).
	\end{aligned}
	\]
	With probability
	\[
	\min\left\{ 1, \frac{\pi(k+1, \alpha', w', \tau', u' \mid y) \, q_{\scriptsize\mbox{D}}(k+1) \, [\sum_{j=1}^{k+1} \ind_{\{0\}}(n_j(\alpha')) ]^{-1} \, (1-w_*)^{k-1} }{ \pi(k, \alpha, w, \tau, u \mid y) \, q_{\scriptsize\mbox{B}}(k) \, g(w_*, \tau_*, u_* \mid k, \alpha, w, \tau, u) }  \right\},
	\]
	set the new state to $(k+1, \alpha', w', \tau', u')$;
	otherwise, keep the old state.
	This move is available only when $k < k_{\scriptsize\mbox{max}}$.
	
	\item {\it D move}:
	Let $\mathsf{E}_k(\alpha) \subset \{1,\dots,k\}$ be the set of $j$'s such that $n_j(\alpha) = 0$.
	Keep the old state if $\mathsf{E}_k(\alpha) = \emptyset$, and follow the procedure below otherwise.
	Randomly and uniformly select $\ell$ from $\mathsf{E}_k(\alpha)$.
	Let $\alpha' = (\alpha'_1, \dots, \alpha'_n)$ be such that, for $i = 1,\dots,n$, $\alpha'_i = \alpha_i$ if $\alpha_i < \ell$ and $\alpha'_i = \alpha_i - 1$ if $\alpha_i > \ell$.
	Let
	\[
	\begin{aligned}
		&w' = (w_1/(1-w_{\ell}), \dots, w_{\ell-1}/(1-w_{\ell}), w_{\ell+1}/(1-w_{\ell}), w_k/(1-w_{\ell})), \\
		&\tau' = (\tau_1, \dots, \tau_{\ell-1}, \tau_{\ell+1}, \dots, \tau_k), \quad u' = (u_1, \dots, u_{\ell-1}, u_{\ell+1}, \dots, u_k).
	\end{aligned}
	\]
	With probability
	\[
	\min \left\{ 1, \frac{\pi(k-1, \alpha', w', \tau', u' \mid y) \, q_{\scriptsize\mbox{B}}(k-1) \, g(w_{\ell}, \tau_{\ell}, u_{\ell} \mid k-1, \alpha', w', \tau', u') }{ \pi(k, \alpha, w, \tau, u \mid y) \,  q_{\scriptsize\mbox{D}}(k) \, [\sum_{j=1}^k \ind_{\{0\}}(n_j(\alpha)) ]^{-1} \, (1-w_{\ell})^{k-2}}  \right\},
	\]
	set the new state to $(k-1, \alpha', w', \tau', u')$; otherwise, keep the old state.
	This move is available only when $k > 1$.
\end{itemize}

\begin{figure} 
	\begin{center}
		\includegraphics[width=0.48\textwidth]{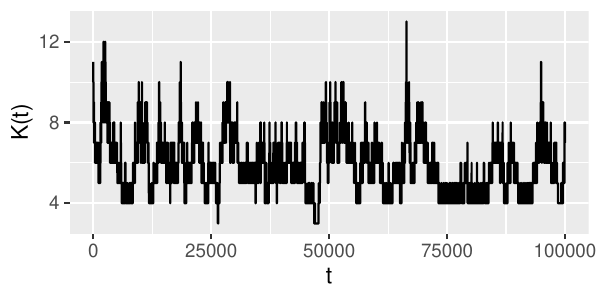} \; \includegraphics[width=0.48\textwidth]{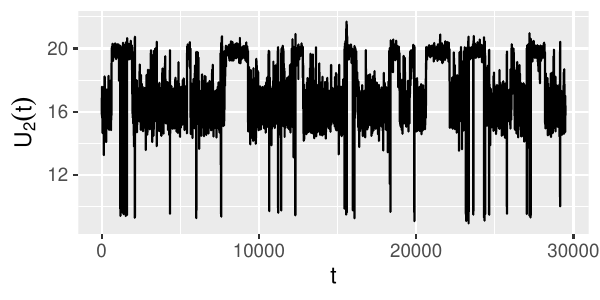} \\
		\includegraphics[width=0.48\textwidth]{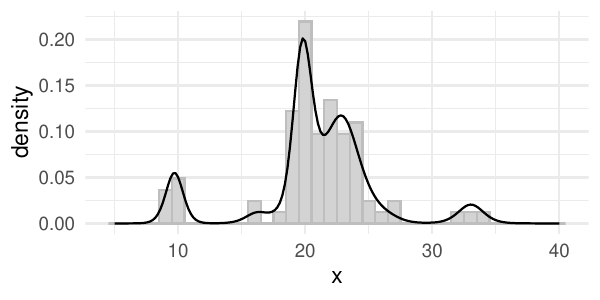} \;
		\includegraphics[width=0.48\textwidth]{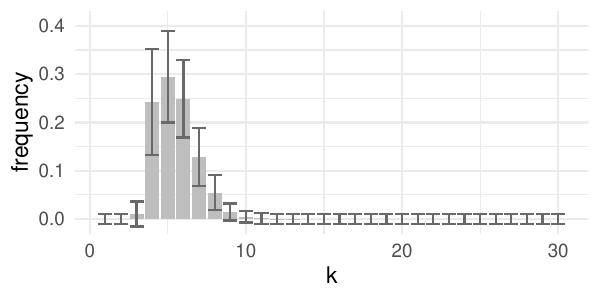}
		\caption{Top left: trace plot of $K(t)$; top right: trace plot of $U_2(t)$ when $K(t) = 5$;
			bottom left: histogram and predictive density for the galaxy data set; 
			bottom right: estimated posterior probabilities of $K = k$ and their 95\% simultaneous confidence intervals.
			Hyperparameters: $k_{\scriptsize\mbox{max}} = 30$, $\gamma = 2$, $b = 2$, $c = 3$, $\tau_0 = 1000$, $u_0 = 20$.
			$f_K(k) \propto 1/2^k$. 
		} \label{fig:galaxy}
	\end{center}
\end{figure}

The resultant chain has $\Pi$ as its stationary distribution due to the reversible jump construction.
On the other hand, the Metropolis-reordering algorithm is not classified as a well-known reversible algorithm, and it remains unclear whether the U move type is reversible or positive semi-definite.

For illustration, the algorithm is applied to the galaxy data set described by \cite{roeder1990density} and studied by \cite{richardson1997bayesian}.
In the B move, $w_*$ is drawn from the beta distribution with parameters $1$ and $n$, $\tau_*$ is drawn from the inverse gamma distribution with shape parameter~$c$ and scale parameter~$b$, and $u_*$ is drawn from the $\mbox{N}(0, \tau_0 \tau_*)$ distribution.
Following \cite{green1995reversible}, the birth and death probabilities are set to
\[
q_{\scriptsize\mbox{B}}(k) = \frac{1}{3} \min \left\{ 1, \frac{f_K(k+1)}{f_K(k)} \right\}, \, q_{\scriptsize\mbox{D}}(k) = \frac{1}{3} \min \left\{ 1, \frac{f_K(k-1)}{f_K(k)} \right\}.
\]
The trans-dimensional chain $(K(t), A(t), W(t), T(t), U(t))_{t=1}^{\infty}$ is simulated for $10^5$ iterations.
The empirical performance of the algorithm is shown in Figure \ref{fig:galaxy}.
The predictive density evaluated at a point~$x$ is the sample average of $\sum_{j=1}^{K(t)} W_j(t) f(x \mid U_j(t), T_j(t))$, where $f(\cdot \mid u, \tau)$ denotes the density of the $\mbox{N}(u,\tau)$ distribution.
We can see that the sampler appears to perform well empirically, especially in terms of within-model moves.

\subsubsection{Convergence analysis}

For $k \in \mathsf{K}$, let $P_k$ be the Mtk associated with the Metropolis re-ordering algorithm targeting $\Phi_k$.
We prove the following lemma in Section \ref{app:reordering-geo} of \ref{s1}.
The proof is constructed based on the fact that, in the Metropolis re-ordering algorithm, the next state depends on the current state $(\alpha,w,\tau,u)$ only through $\alpha$, which takes value in a finite set.

{
\begin{lemma} \label{lem:reordering-geo}
	For $k \in \mathsf{K}$, $\|P_k^2\|_{\Phi_k} < 1$.
\end{lemma}
}

With Lemma \ref{lem:reordering-geo} in hand, it is now straightforward to establish the geometric ergodicity of the reversible jump algorithm.

\begin{proposition}
	Suppose that $q_U(k) > 0$ for $k \in \mathsf{K}$, $q_B(k) > 0$ for $k < k_{\scriptsize\mbox{max}}$, and $q_D(k) > 0$ for $k > 0$.
	Then the reversible jump chain is $L^2(\Pi)$ geometrically convergent and $\Pi$-a.e. geometrically ergodic.
\end{proposition}

\begin{proof}
	Apply Theorem \ref{thm:main}.
	Let $P$ be the Mtk of the reversible jump chain.
	For $k \in \mathsf{K}$, (i) and (iii) in \ref{H1} hold with $c_k = q_U(k)$.
	By Lemma \ref{lem:reordering-geo}, (ii) in \ref{H1} holds with $t_0 = 2$.
	
	To verify \ref{H2}, note that $\bar{P}(k,\{k'\}) > 0$ whenever $|k-k'| \leq 1$.
	It is then evident that $\bar{P}$ is irreducible.
	
	Thus, $P$ is $L^2(\Pi)$ geometrically convergent.
	By Theorem~1 of \cite{roberts2001geometric}, it is $\Pi$-a.e. geometrically ergodic.
\end{proof}

\begin{remark}
	The reversible jump algorithm can be further improved by adding more sophisticated between-model move types such as split and merge \citep{richardson1997bayesian,zhang2004learning}.
	Geometric convergence is preserved as long as the between-model movements remain irreducible.
\end{remark}

Geometric ergodicity allows us to construct 95\% simultaneous Wald confidence intervals for the posterior probabilities of $K = k$ for $k = 1,\dots, k_{\scriptsize\mbox{max}}$.
This is shown for the galaxy data set in Figure \ref{fig:galaxy}.
We have added $(\log m) \sqrt{1/b_m^2+b_m/m}$ to the estimated asymptotic variances of the estimated probabilities, as in Section \ref{sssec:binary-example}.

{\section{Discussion} \label{sec:discussion} }

$L^2(\Pi)$ geoemtrically convergence implies $\Pi$-a.e. geometric ergodicity.
Harris recurrence, a regularity condition commonly used in MCMC analysis, can be enforced provided that we restrict our attention to some absorbing set \citep[][Theorem 9.0.1]{meyn2012markov}.
See also \cite{roberts2006harris} for conditions for Harris ergodicity in the context of trans-dimensional chains.

Under geometric ergodicity, one may utilize methods from e.g., \cite{jones2006fixed,vats2019multivariate} to construct confidence regions for uncertainty quantification.
In our examples, we used Bonferroni correction to construct multiple Wald confidence intervals, but more sophisticated methods exist \citep{vats2019multivariate,robertson2020assessing}.

{
Some of the formulas in the proof of Lemma \ref{lem:decomposition} (found in Section \ref{app:decomposition} of \ref{s1}) can be formulated in terms of Dirichlet forms.
Dirichlet forms may be used to study Markov chains in terms of the conductance \citep{lawler1988bounds}, Peskun-Tierney ordering \citep{andrieu2019peskun}, and functional inequalities \citep{power2024weak}.
}

An obvious avenue for future research is obtaining practical quantitative convergence bounds for trans-dimensional chains.
In particular, ways of computing or estimating $\mbox{Gap}_{\bar{\Pi}}(\overline{P^*P})$ and $\mbox{Gap}_{\bar{\Pi}}(\bar{P})$ would be useful for selecting good proposal distributions in a reversible jump algorithm. 
Whether our convergence bounds can be further sharpened and extended to the case where $\mathsf{K}$ is infinite is also of interest.

\vspace{0.5cm}

\noindent{\bf Acknowledgment:}
We thank the Editor, an anonymous Associate Editor, and two anonymous Referees for their helpful comments.

\bigskip
\begin{center}
{\large\bf SUPPLEMENTARY MATERIALS}
\end{center}

{\bf \namedlabel{s1}{Supplement I}: Minor Results and Technical Proofs. } This document contains a proof for the assertion that $\Pi$-irreducibility implies \ref{H2}, proofs for Lemmas \ref{lem:decomposition}, \ref{lem:reordering}, \ref{lem:reordering-geo}, and simulation results concerning Monte Carlo confidence intervals.

{\bf \namedlabel{s2}{Supplement II}: Autoregression with Laplace errors. } This document contains an application of the theory and methods herein to an autoregressive model with Laplace errors and unknown model order.

\bibliographystyle{Chicago}

\bibliography{qinbib}

\newpage

\begin{center} 
	{\bf \LARGE Supplement I: Minor Results and Technical Proofs}
\end{center}

\appendix

\section{$\Pi$-irreducibility implies \ref{H2}} \label{app:h2}

Assume that $P$ is $\Pi$-irreducible.
Fix $k, k' \in \mathsf{K}$.
By Proposition 4.2 of \cite{meyn2012markov}, for $z \in \Z_k$,
\[
\sum_{t=1}^{\infty} 2^{-t} P^t((k,z), \{k'\} \times \Z_{k'}) > 0.
\]
By the monotone convergence theorem, 
\[
\sum_{t=1}^{\infty} 2^{-t} \int_{\Z_k} \Phi_k(\df z) P^t((k,z), \{k'\} \times \Z_{k'}) > 0.
\]
That is, one can find a positive integer $s$ such that 
\[
\int_{\Z_k} \Phi_k(\df z) P^s((k,z), \{k'\} \times \Z_{k'}) > 0.
\]
Let $\{K(t), Z(t)\}_{t=0}^{\infty}$ be a chain associated with~$P$ such that $K(0) = k$ and $Z(t) \sim \Phi_k$.
Then the above display reads $\mbox{Pr}(K(s) = k') > 0$.
We may partition the event $[K(s) = k']$ according to the path of $K(t)$ for $t = 0,\dots,s$.
It follows that there exist $k(0), \dots, k(s) \in \mathsf{K}$ such that $k(0) = k$, $k(s) = k'$, and that
\[
\mbox{Pr} \left( K(t) = k(t) \text{ for } t = 0,\dots,s \right) > 0.
\] 
In particular, for $t \in \{0,\dots,s-1\}$, $\mbox{Pr}(K(t) = k(t), \, K(t+1) = k(t+1)) > 0$, i.e.,
\[
\int_{\Z_k} \mu_t(\df z) P((k(t),z), \{k(t+1)\} \times \Z_{k(t+1)} ) > 0,
\]
where $\mu_t(\mathsf{A}) = \int_{\Z_k} \Phi_k(\df z) P^t((k,z), \{k(t)\} \times \mathsf{A})$ for $\mathsf{A} \in \A_{k(t)}$.
For $\mathsf{A} \in \A_{k(t)}$,
\[
\mu_t(\mathsf{A}) \leq \frac{1}{\Psi_k(\Z_k)} \sum_{k'' \in \mathsf{K}} \Psi_{k''}(\df z) P^t((k'',z), \{k(t)\} \times \mathsf{A}) = \frac{\Psi_{k(t)}(\mathsf{A})}{\Psi_k(\Z_k)}.
\]
Then it holds that, for $t \in \{0,\dots,s-1\}$,
\[
\int_{\Z_{k(t)}} \Psi_{k(t)}(\df z) \, P((k(t),z), \{k(t+1)\} \times \Z_{k(t+1)} ) > 0,
\]
i.e., $\bar{P}(k(t), \{k(t+1)\}) > 0$.
This in turn implies that $\bar{P}^s(k,\{k'\}) > 0$.
Thus, $\bar{P}$ is irreducible, i.e., \ref{H2} holds.

\section{Proof of Lemma~\ref{lem:decomposition}} \label{app:decomposition}

Recall some notations: 
$(\Y,\F,\omega)$ is a probability space.
$\mtkfont{S}$ and $\Mtk$ are Mtks that have $\omega$ as their stationary distributions.
The state space~$\Y$ has a disjoint decomposition $\Y = \bigcup_{k \in \mathsf{K}} \Y_k$.
For $k \in \mathsf{K}$, $\F_k$ is the restriction of~$\F$ to $\Y_k$, $\omega_k$ is the normalized restriction of~$\omega$ on $\Y_k$, and $\bar{\omega}(\{k\}) = \omega(\Y_k)$.
The Mtk $\Mtk_k: \Y_k \times \F_k \to [0,1]$ satisfies $\omega_k \Mtk_k = \omega_k$, and $\Mtk(y,\mathsf{B}) \geq c\Mtk_k(y, \mathsf{B})$ for $y \in \Y_k$ and $\mathsf{B} \in \F_k$, where $c \in [0,1]$.
Finally, the kernel
\[
\overline{\mtkfont{S}^* \mtkfont{S}}(k,\{k'\}) = \frac{1}{\omega(\Y_k)} \int_{\Y} \omega(\df y) \mtkfont{S}(y, \Y_k) \mtkfont{S}(y, \Y_{k'}).
\]
defines a positive semi-definite operator on $L_0^2(\bar{\omega})$.
The goal is to show that
\[
1 - \|\Mtk \mtkfont{S}^*\|_{\omega}^2 \geq c^2 \left(1 - \sup_{k \in \mathsf{K}} \|\Mtk_k\|_{\omega_k}^2 \right) (1 - \|\overline{\mtkfont{S}^* \mtkfont{S}}\|_{\bar{\omega}} ).
\]

\begin{proof}
	Let $f \in L_0^2(\omega)$ be arbitrary.
	Then
	\begin{equation} \label{eq:fnormdecomp-1}
		\|f\|_{\omega}^2 - \|\Mtk \mtkfont{S}^* f\|_{\omega}^2 = \|f\|_{\omega}^2 - \|\mtkfont{S}^*f\|_{\omega}^2 + \|\mtkfont{S}^*f\|_{\omega}^2 - \|\Mtk \mtkfont{S}^* f\|_{\omega}^2.
	\end{equation}
	For $y \in \Y$ and $\mathsf{B} \in \F$, let
	\[
	\hat{\Mtk}(y,\mathsf{B}) = \sum_{k \in \mathsf{K}} 1_{\Y_k}(y) \Mtk_k(y, \mathsf{B} \cap \Y_k).
	\]
	Then $\omega \hat{\Mtk} = \omega$, and $\Mtk(y, \cdot) \geq c \hat{\Mtk}(y, \cdot)$ for $y \in \Y$.
	Thus, there exists an Mtk $\mtkfont{R}: \Y \times \F \to [0,1]$ such that $\omega \mtkfont{R} = \omega$ and $\Mtk = c\hat{\Mtk} + (1-c) \mtkfont{R}$.
	{ By the Cauchy-Schwarz inequality and the fact that the norms of Markov operators are no greater than one,}
	\begin{equation} \nonumber
		\begin{aligned}
			\|\Mtk \mtkfont{S}^*f\|_{\omega}^2 =& c^2 \|\hat{\Mtk} \mtkfont{S}^* f\|_{\omega}^2 + (1-c)^2 \|\mtkfont{R} \mtkfont{S}^* f\|_{\omega}^2 + 2 c(1-c)  \langle \hat{\Mtk} \mtkfont{S}^* f, \mtkfont{R} \mtkfont{S}^* f \rangle_{\omega}  \\
			\leq & c^2 \|\hat{\Mtk} \mtkfont{S}^* f\|_{\omega}^2 + (1-c^2) \|\mtkfont{S}^* f\|_{\omega}^2.
		\end{aligned}
	\end{equation}
	It follows that
	\begin{equation} \label{ine:fnormdecomp-2}
		\|\mtkfont{S}^*f\|_{\omega}^2 - \|\Mtk \mtkfont{S}^* f\|_{\omega}^2 \geq c^2 \left( \|\mtkfont{S}^*f\|_{\omega}^2 - \|\hat{\Mtk } \mtkfont{S}^* f\|_{\omega}^2 \right).
	\end{equation}
	
	Define an operator $E: L_0^2(\omega) \to L_0^2(\omega)$ as
	\begin{equation} \nonumber
		Ef(y) = \sum_{k \in \mathsf{K}} 1_{\Y_k}(y) \frac{1}{\omega(\Y_k)} \int_{\Y_k} \omega(\df y') f(y') = \sum_{k \in \mathsf{K}} 1_{\Y_k}(y) \int_{\Y_k} \omega_k(\df y') f(y').
	\end{equation}
	Then~$E$ is an orthogonal projection: $E^2 = E^* = E$.
	Its range is the set of $L^2_0(\omega)$ functions that are constant on any given $\Y_k$.
	It is easy to see that $\hat{\Mtk }E = E$.
	In fact, $E\hat{\Mtk } = E$ as well.
	To see this, note that, for $g \in L_0^2(\omega)$,
	\[
	\begin{aligned}
		E \hat{\Mtk } g(y) &= \sum_{k \in \mathsf{K}} 1_{\Y_k}(y) \int_{\Y_k} \omega_k(\df y') \int_{\Y_k} \Mtk _k(y', \df y'') g(y'') \\
		&= \sum_{k \in \mathsf{K}} 1_{\Y_k}(y) \int_{\Y_k} \omega_k(\df y'') g(y'').
	\end{aligned}
	\]
	It then follows that 
	\begin{equation} \label{eq:ET}
		\begin{aligned}
			\|\mtkfont{S}^*f\|_{\omega}^2 - \|\hat{\Mtk } \mtkfont{S}^* f\|_{\omega}^2 =& \|[E + (I-E)] \mtkfont{S}^*f\|_{\omega}^2 - \|\hat{\Mtk } [E + (I-E)] \mtkfont{S}^* f\|_{\omega}^2 \\
			=& \|E\mtkfont{S}^*f \|_{\omega}^2 + \|(I-E) \mtkfont{S}^* f\|_{\omega}^2 - \|\hat{\Mtk } E \mtkfont{S}^* f \|_{\omega}^2 - \|\hat{\Mtk } (I-E) \mtkfont{S}^* f \|_{\omega}^2 - \\
			& 2 \langle \hat{\Mtk } E \mtkfont{S}^*f, \hat{\Mtk } (I - E)  \mtkfont{S}^* f \rangle_{\omega} \\
			=& \|(I-E) \mtkfont{S}^* f\|_{\omega}^2 - \|\hat{\Mtk } (I-E) \mtkfont{S}^* f \|_{\omega}^2.
		\end{aligned}
	\end{equation}
	Here, $I$ is the identity operator.
	{ We have used the following properties of an orthogonal projection: $E = E^*$ and $(I-E)E = E(I-E) = 0$.}

	For $k \in \mathsf{K}$, let $M_k \mtkfont{S}^* f: \Y_k \to \mathbb{R}$ be such that
	\[
	M_k \mtkfont{S}^* f (y) = (I-E) \mtkfont{S}^* f(y) = \mtkfont{S}^* f(y) - \int_{\Y_k} \omega_k(\df y') \mtkfont{S}^* f(y')
	\]
	for $y \in \Y_k$.
	Then $M_k\mtkfont{S}^* f$ is in $L_0^2(\omega_k)$, and
	\[
	\|(I - E) \mtkfont{S}^* f \|_{\omega}^2 = \sum_{k \in \mathsf{K}} \int_{\Y_k} \omega(\df y) \, [M_k \mtkfont{S}^* f(y)]^2 = \sum_{k \in \mathsf{K}} \omega(\Y_k) \|M_k \mtkfont{S}^* f \|_{\omega_k}^2.
	\]
	Moreover,
	\begin{equation} \label{ine:fnormdecomp-3}
		\begin{aligned}
			\|\hat{\Mtk } (I-E) \mtkfont{S}^* f \|_{\omega}^2 &= \sum_{k \in \mathsf{K}} \int_{\Y_k} \omega(\df y) \, [\Mtk _k M_k \mtkfont{S}^* f(y)]^2 \\
			&= \sum_{k \in \mathsf{K}} \omega(\Y_k) \|\Mtk _k M_k \mtkfont{S}^* f \|_{\omega_k}^2 \\
			&\leq \left( \sup_{k \in \mathsf{K}} \|\Mtk _k\|_{\omega_k}^2 \right) \|(I - E) \mtkfont{S}^* f \|_{\omega}^2 .
		\end{aligned}
	\end{equation}
	
	Note that $\|\mtkfont{S}^*\|_{\omega} = \|\mtkfont{S} \|_{\omega} \leq 1$, so $\|f\|_{\omega}^2 - \|\mtkfont{S}^*f\|_{\omega}^2 \geq 0$.
	Combining this fact with~\eqref{eq:fnormdecomp-1} to~\eqref{ine:fnormdecomp-3} yields
	\[
	\begin{aligned}
		\|f\|_{\omega}^2 - \|\Mtk \mtkfont{S}^* f\|_{\omega}^2 & \geq \|f\|_{\omega}^2 - \|\mtkfont{S}^*f\|_{\omega}^2 + c^2 \left( 1 - \sup_{k \in \mathsf{K}} \|\Mtk _k\|_{\omega_k}^2 \right) \|(I - E) \mtkfont{S}^* f \|_{\omega}^2 \\
		& \geq c^2 \left( 1 - \sup_{k \in \mathsf{K}} \|\Mtk _k\|_{\omega_k}^2 \right) \left[ \|f\|_{\omega}^2 - \|\mtkfont{S}^*f\|_{\omega}^2 + \|(I - E) \mtkfont{S}^* f \|_{\omega}^2  \right] \\
		&= c^2 \left( 1 - \sup_{k \in \mathsf{K}} \|\Mtk _k\|_{\omega_k}^2 \right) \left[ \|f\|_{\omega}^2 - \|E \mtkfont{S}^*f\|_{\omega}^2  \right] \\
		&\geq c^2 \left( 1 - \sup_{k \in \mathsf{K}} \|\Mtk _k\|_{\omega_k}^2 \right) (1 - \|E\mtkfont{S}^*\|_{\omega}^2) \|f\|_{\omega}^2.
	\end{aligned}
	\]
	Since~$f$ is arbitrary,
	\[
	1 - \|\Mtk \mtkfont{S}^*\|_{\omega}^2 \geq c^2 \left( 1 - \sup_{k \in \mathsf{K}} \|\Mtk _k\|_{\omega_k}^2 \right) (1 - \|E\mtkfont{S}^*\|_{\omega}^2).
	\]
	
	To complete the proof, it suffices to show that $\|E\mtkfont{S}^*\|_{\omega}^2 = \|\overline{\mtkfont{S}^* \mtkfont{S}}\|_{\bar{\omega}}$.
	Note that $\|ES^*\|_{\omega} = \|(E\mtkfont{S}^*)^*\|_{\omega} = \|\mtkfont{S}E\|_{\omega}$.
	Denote the range of~$E$ by~$H$.
	Then
	\begin{equation} \label{eq:SEnorm-1}
		\|E\mtkfont{S}^*\|_{\omega}^2 = \|\mtkfont{S}E\|_{\omega}^2 = \sup_{f \in L_0^2(\omega) \setminus \{0\}} \frac{\|\mtkfont{S}Ef\|_{\omega}^2}{\|Ef\|_{\omega}^2 + \|(I-E)f\|_{\omega}^2} =  \sup_{f \in H \setminus \{0\}} \frac{\|\mtkfont{S} f\|_{\omega}^2}{\|f\|_{\omega}^2}.
	\end{equation}
	The transformation that maps a function $g: \mathsf{K} \to \mathbb{R}$ to $\sum_{k \in \mathsf{K}} g(k) 1_{\Y_k}$
	defines an isomorphism from $L_0^2(\bar{\omega})$ to~$H$.
	Call this transformation~$U$.
	Because~$U$ is unitary,
	\begin{equation} \label{eq:SEnorm-2}
		\sup_{f \in H \setminus \{0\}} \frac{\|\mtkfont{S}f\|_{\omega}^2}{\|f\|_{\omega}^2} =  \sup_{g \in L_0^2(\bar{\omega})  \setminus \{0\}} \frac{\|\mtkfont{S} Ug\|_{\omega}^2}{\|g\|_{\bar{\omega}}^2}.
	\end{equation} 
	For $g \in L_0^2(\bar{\omega})$,
	\begin{equation} \label{eq:SEnorm-3}
		\|\mtkfont{S} Ug\|_{\omega}^2 = \int_{\Y} \omega(\df y) \left[ \sum_{k \in \mathsf{K}} \mtkfont{S}(y, \Y_k) g(k) \right]^2 = \langle g, \overline{\mtkfont{S}^*\mtkfont{S}} g \rangle_{\bar{\omega}}.
	\end{equation}
	Note that $\overline{\mtkfont{S}^* \mtkfont{S}}$ defines a positive semi-definite operator on $L_0^2(\bar{\omega})$.
	Then, by~\eqref{eq:SEnorm-1},~\eqref{eq:SEnorm-2}, and~\eqref{eq:SEnorm-3}, $\|E\mtkfont{S}^*\|_{\omega}^2 = \|\overline{\mtkfont{S}^* \mtkfont{S}}\|_{\bar{\omega}}$ as desired.
\end{proof}


{
	Finally, we derive the assertion in Remark \ref{rem:lower}:
	$
	1 - \|\Mtk\|_{\omega}^2 \leq \mbox{Gap}_{\bar{\omega}}(\overline{\Mtk^* \Mtk}) \leq 1 - \|\bar{\mtkfont{T}}\|_{\bar{\omega}}^2.
	$
	The result trivially holds if $|\mathsf{K}| = 1$ since, in this case, $ \mbox{Gap}_{\bar{\omega}}(\overline{\Mtk^* \Mtk}) =  1$ and $\|\bar{\Mtk}\|_{\bar{\omega}} = 0$.
	Assume that $|\mathsf{K}| > 1$.
	Then the range of the projection $E$ is not just $\{0\}$.
	By \eqref{eq:SEnorm-2} and \eqref{eq:SEnorm-3},
	\[
	\|\mtkfont{S}\|_{\omega}^2 \geq \sup_{f \in H \setminus \{0\}} \frac{\|\mtkfont{S} f\|_{\omega}^2}{\|f\|_{\omega}^2} = \|\overline{\mtkfont{S}^* \mtkfont{S}}\|_{\bar{\omega}} = 1 - \mbox{Gap}_{\bar{\omega}} ( \overline{\mtkfont{S}^* \mtkfont{S}} ).
	\]
	Moreover, by the Cauchy-Schwarz inequality,
	\[
	\begin{aligned}
		\|\bar{\mtkfont{S}}\|_{\bar{\omega}}^2 &= \sup_{g \in L_0^2(\bar{\omega})  \setminus \{0\}} \|g\|_{\bar{\omega}}^{-2} \sum_{k \in \mathsf{K}} \bar{\omega}(\{k\}) \left[  \int_{\Y_k} \omega_k(\df y) \sum_{k' \in \mathsf{K}} S(y, \Y_{k'}) g(k') \right]^2 \\
		&\leq \sup_{g \in L_0^2(\bar{\omega}) \setminus \{0\}} \|g\|_{\bar{\omega}}^{-2} \sum_{k \in \mathsf{K}} \bar{\omega}(\{k\})   \int_{\Y_k} \omega_k(\df y) \left[ \sum_{k' \in \mathsf{K}} S(y, \Y_{k'}) g(k') \right]^2 \\
		&= \sup_{g \in L_0^2(\bar{\omega}) \setminus \{0\}} \frac{ \langle g, \overline{\mtkfont{S}^*\mtkfont{S}} g \rangle_{\bar{\omega}} }{\|g\|_{\bar{\omega}}^2} \\
		&= \|\overline{S^*S}\|_{\bar{\omega}}.
	\end{aligned}
	\]
	Letting $\mtkfont{S} = \Mtk$ yields the desired inequalities.
}

{
	\section{Proof of Lemma~\ref{lem:reordering}} \label{app:reordering-invariant}
}

Let $k \in \mathsf{K}$ be fixed.
First we note that the un-normalized posterior density associated with $\Phi_k$ can be written into the form
\[
\phi_k(\alpha, w, \tau, u) = C_k \tilde{\psi}_k(\alpha, w, \tau, u) \, k! \, \ind_{\mathsf{G}_k}(u),
\]
where $C_k$ is a normalizing constant, and 
\[
\begin{aligned}
	\tilde{\psi}_k(\alpha, w, \tau, u) =& \left\{ \prod_{i=1}^n w_{\alpha_i} \frac{1}{\sqrt{2\bm{\pi} \tau_{\alpha_i} } } \exp\left[ - \frac{1}{2 \tau_{\alpha_i}} (y_i - u_{\alpha_i})^2 \right] \right\} \\
	& \left\{ \prod_{j=1}^k w_j^{\gamma - 1} \frac{b^c}{\Gamma(c)} \tau_j^{-c - 1} e^{-b/\tau_j} \frac{1}{\sqrt{2 \bm{\pi} \tau_0 \tau_j}} \exp\left[ - \frac{(u_j-u_0)^2}{2\tau_0 \tau_j}  \right] \ind_{(0,\infty)}(\tau_j) \right\} \\
	&  \frac{\Gamma(k\gamma)}{\Gamma(\gamma)^k} \ind_{\mathsf{S}_k}(w)
\end{aligned}
\]
It can be checked that $C_k \tilde{\psi}_k(\cdot)$ is a probability density function that is invariant under label switching.
A label switch associated with a permutation $\sigma: \{1,\dots,k\} \to \{1,\dots,k\}$ is a one-to-one function on $\tilde{\Z}_k := \{1,\dots,k\}^n \times \mathsf{S}_k \times (0,\infty)^k \times \mathbb{R}^k$ that results from re-labeling the $j$th Gaussian component in the mixture model as the $\sigma(j)$'th component for $j = 1,\dots,k$.
To be precise, when applied to $(\alpha, w, \tau, u) \in \tilde{\Z}_k$, it maps $\alpha = (\alpha_1, \dots, \alpha_n)$ to $L_{\sigma,1} (\alpha) := (\sigma(\alpha_1), \dots, \sigma(\alpha_n))$, $w = (w_1, \dots, w_k)$ to $L_{\sigma,2}(w) := (w_{\sigma^{-1}(1)}, \dots, w_{\sigma^{-1}(k)})$, $\tau$ to $L_{\sigma,2}(\tau)$ and $u$ to $L_{\sigma,2}(u)$.
We shall use $\tilde{\Phi}_k$ to denote the distribution associated with $C_k \tilde{\psi}_k(\cdot)$.

In the Metropolis re-ordering algorithm, with some abuse of notation, denote the current state by $(A,W,T,U)$, and denote the random elements generated by steps 1 to 6 by $A^{(0)}$, $A''$, $W''$, $T''$, $U''$, $(A',W',T',U')$, respectively.
Assume that $(A,W,T,U) \sim \Phi_k$.
The goal is to show that $(A',W',T',U') \sim \Phi_k$.

Denote the collection of permutations on $\{1,\dots,k\}$ by $\Sigma_k$.
By the invariance of $\tilde{\psi}_k$ under label switching, for $\mathsf{B} \subset \{1,\dots,k\}^n$, 
\[
\begin{aligned}
	\mbox{Pr}(A^{(0)} \in \mathsf{B}) 
	=& \frac{1}{k!} \sum_{\sigma \in \Sigma_k} \mbox{Pr}(L_{\sigma,1}(A) \in \mathsf{B}) \\
	=&  C_k \sum_{\sigma \in \Sigma_k} \sum_{\alpha \in \{1,\dots,k\}^n} \int_{\mathsf{S}_k} \int_{(0,\infty)^k} \int_{\mathbb{R}^k} \ind_{\mathsf{B}}(L_{\sigma,1}(\alpha)) \, \ind_{\mathsf{G}_k}(u) \, \tilde{\psi}_k(\alpha, w, \tau, u) \, \df u \, \df \tau \, \df w \\
	=&  C_k \sum_{\sigma \in \Sigma_k} \sum_{\alpha \in \{1,\dots,k\}^n} \int_{\mathsf{S}_k} \int_{(0,\infty)^k} \int_{\mathbb{R}^k} \ind_{\mathsf{B}}(L_{\sigma,1}(\alpha)) \, \ind_{L_{\sigma,2}(\mathsf{G}_k)}(L_{\sigma,2}(u)) \\
	& \tilde{\psi}_k(L_{\sigma,1}(\alpha), L_{\sigma,2}(w), L_{\sigma,2}(\tau), L_{\sigma,2}(u)) \, \df u \, \df \tau \, \df w \\
	=&  C_k \sum_{\sigma \in \Sigma_k} \sum_{\tilde{\alpha} \in \{1,\dots,k\}^n} \int_{\mathsf{S}_k} \int_{(0,\infty)^k} \int_{\mathbb{R}^k} \ind_{\mathsf{B}}(\tilde{\alpha}) \, \ind_{L_{\sigma,2} (\mathsf{G}_k)}(\tilde{u}) \, \tilde{\psi}_k(\tilde{\alpha}, \tilde{w}, \tilde{\tau}, \tilde{u}) \, \df \tilde{u} \, \df \tilde{\tau} \, \df \tilde{w} \\
	=& C_k \sum_{\tilde{\alpha} \in \{1,\dots,k\}^n} \int_{\mathsf{S}_k} \int_{(0,\infty)^k} \int_{\mathbb{R}^k} \ind_{\mathsf{B}}(\tilde{\alpha}) \, \tilde{\psi}_k(\tilde{\alpha}, \tilde{w}, \tilde{\tau}, \tilde{u}) \, \df \tilde{u} \, \df \tilde{\tau} \, \df \tilde{w}.
\end{aligned}
\]
Note that we have utilized the fact that the determinant of $L_{\sigma,2}$ is~1, and that $\sum_{\sigma} \ind_{L_{\sigma,2}(\mathsf{G}_k)}(\tilde{u}) = 1$ for $\tilde{u} \in \mathbb{R}^k$ outside a Lebesgue measure zero set.
The above display shows that $A^{(0)}$ has the same distribution as $\tilde{A}$ if $(\tilde{A}, \tilde{W}, \tilde{T}, \tilde{U}) \sim \tilde{\Phi}_k$.
In fact, it is easy to see that the probability mass function of this distribution is precisely $\tilde{\pi}_A(\cdot \mid k, y)$.
Step 2 in the Metropolis re-ordering algorithm is but a sequence of Metropolis Hastings steps targeting this distribution.
Thus, the distribution of $A''$ remains the same.
It can be checked that, given $A''$,  $(W'', T'', U'')$ follows precisely the conditional distribution of $(\tilde{W}, \tilde{T}, \tilde{U})$ given $\tilde{A}$ if $(\tilde{A}, \tilde{W}, \tilde{T}, \tilde{U}) \sim \tilde{\Phi}_k$.
Therefore, $(A'', W'', T'', U'') \sim \tilde{\Phi}_k$.
Again by the label-switching invariance, for $\mathsf{B} \in \Z_k$,
\[
\begin{aligned}
	& \mbox{Pr}((A',W',T',U') \in \mathsf{B}) \\
	=& \sum_{\sigma \in \Sigma_k} \mbox{Pr} ( (L_{\sigma,1}(A''), L_{\sigma,2}(W''), L_{\sigma,2}(T''), L_{\sigma,2}(U'')) \in B, \; L_{\sigma,2}(U'') \in \mathsf{G}_k ) \\
	=& C_k \sum_{\sigma \in \Sigma_k} \sum_{\alpha'' \in \{1,\dots,k\}^n} \int_{\mathsf{S}_k} \int_{(0,\infty)^k} \int_{\mathbb{R}^k} \ind_{\mathsf{B}}(L_{\sigma,1}(\alpha''), L_{\sigma,2}(w''), L_{\sigma,2}(\tau''), L_{\sigma,2}(u'') ) \, \ind_{\mathsf{G}_k}(L_{\sigma,2}(u'')) \\
	& \tilde{\psi}_k(\alpha'', w'', \tau'', u'') \, \df u'' \, \df \tau'' \, \df w'' \\
	=& C_k \sum_{\sigma \in \Sigma_k} \sum_{\alpha' \in \{1,\dots,k\}^n} \int_{\mathsf{S}_k} \int_{(0,\infty)^k} \int_{\mathbb{R}^k} \ind_{\mathsf{B}} (\alpha', w', \tau', u') \, \ind_{\mathsf{G}_k}(u') \, \tilde{\psi}_k(\alpha', w', \tau', u') \, \df u' \, \df \tau' \, \df w' \\
	=& \Phi_k(\mathsf{B}).
\end{aligned}
\]
This completes the proof.

{
	\section{Proof of Lemma~\ref{lem:reordering-geo}} \label{app:reordering-geo}
}

Let $f \in L_0^2(\Phi_k)$ be arbitrary.
Then by Lemma \ref{lem:reordering}, $P_k f \in L_0^2(\Phi_k)$, and $\|P_k f\|_{\Phi_k} \leq \|f\|_{\Phi_k}$.
Denote by $\mathcal{H}_A$ the subspace of functions in $L_0^2(\Phi_k)$ that depend on $(\alpha, w, \tau, u)$ only through~$\alpha$.
Then $P_k f \in \mathcal{H}_A$ since the measure $P_k((\alpha,w,\tau,u), \cdot)$ depends on $(\alpha,w,\tau,u)$ only through~$\alpha$.
It follows that $\|P_k^2 f\|_{\Phi_k} \leq \|P_k|_{\mathcal{H}_A}\|_{\Phi_k} \|P_k f\|_{\Phi_k} \leq \|P_k|_{\mathcal{H}_A}\|_{\Phi_k} \|f\|_{\Phi_k}$, where $P_k |_{\mathcal{H}_A}: \mathcal{H}_A \to \mathcal{H}_A$ is $P_k$ restricted to $\mathcal{H}_A$, and $\|P_k|_{\mathcal{H}_A}\|_{\Phi_k} = \sup_g \|P_k g\|_{\Phi_k}/\|g\|_{\Phi_k}$ where the supremum is taken over $\mathcal{H}_A \setminus \{0\}$.
To prove $\|P_k^2\|_{\Phi_k} < 1$, it suffice to show that $\|P_k |_{\mathcal{H}_A} \|_{\Phi_k} < 1$.
Since $\mathcal{H}_A$ is finite dimensional, bounded linear operators on $\mathcal{H}_A$ must be norm-attaining.
Thus, to show that  $\|P_k |_{\mathcal{H}_A} \|_{\Phi_k} < 1$, it suffices to show that $\|P_k g \|_{\Phi_k} < \|g\|_{\Phi_k}$ for any $g \in \mathcal{H}_A \setminus \{0\}$.

Let $g \in \mathcal{H}_A \setminus \{0\}$ be arbitrary.
Contrary to what we wish to prove, assume that $\|P_k g \|_{\Phi_k} = \|g\|_{\Phi_k}$.
Abusing notations we write $g(\alpha,w,\tau,u)$ as $g(\alpha)$.
Then $P_k g(\alpha) := P_k g(\alpha, w, \tau, u)$ can be written as $\sum_{\alpha' \in \{1,\dots,k\}^n} p(\alpha,\alpha') g(\alpha')$, where $p(\alpha,\alpha') = P_k((\alpha,w,\tau,u), \{\alpha'\} \times \mathsf{S}_k \times (0,\infty)^n \times \mathsf{G}_k)$ for $(w,\tau,u) \in  \mathsf{S}_k \times (0,\infty)^n \times \mathsf{G}_k$.
Let $\pi_A(\alpha) = \int_{\mathsf{S}_k \times (0,\infty)^n \times \mathsf{G}_k} \phi_k(\alpha,w,\tau,u) \, \df(w,\tau,u)$ for $\alpha \in \{1,\dots,k\}^n$, where $\phi_k$ is the density of $\Phi_k$.
Then $\pi_A(\alpha) > 0$ for $\alpha \in \{1,\dots,k\}^n$, and by Lemma \ref{lem:reordering}, $\sum_{\alpha \in \{1,\dots,k\}^n} \pi_A(\alpha) \, p(\alpha, \alpha') = \pi_A(\alpha')$ for $\alpha' \in \{1,\dots,k\}^n$.
One can derive derive
\[
\begin{aligned}
	& \|g\|_{\Phi_k}^2 = \sum_{\alpha \in \{1,\dots,k\}^n} \pi_A(\alpha) \sum_{\alpha' \in \{1,\dots,k\}^n} p(\alpha,\alpha') \, g(\alpha')^2, \\
	& \|P_k g\|_{\Phi_k}^2 = \sum_{\alpha \in \{1,\dots,k\}^n} \pi_A(\alpha) \left[ \sum_{\alpha' \in \{1,\dots,k\}^n} p(\alpha,\alpha') \, g(\alpha') \right]^2.
\end{aligned}
\]
Then, by the Cauchy-Schwarz inequality, $\|P_k g\|_{\Phi_k}^2 < \|g\|_{\Phi_k}^2$ unless, for $\alpha \in \{1,\dots,k\}^n$, the function $\alpha' \mapsto \sqrt{p(\alpha, \alpha') } \, g(\alpha')$ is proportional to $\sqrt{p(\alpha,\alpha')}$, i.e., $g(\alpha')$ is constant on the set $\{\alpha': \, p(\alpha,\alpha') > 0\}$.
By the construction of the Metropolis re-ordering algorithm, $p(\alpha,\alpha') > 0$ for $\alpha, \alpha' \in \mathsf{K}$.
Thus, $g$ must be a constant function on $\mathsf{K}$.
But this cannot happen, as $g$ is nonzero, and $\mathcal{H}_A$, as a subset of $L_0^2(\Phi_k)$, does not contain nonzero constant functions.
Therefore, $\|P_k g \|_{\Phi_k} < \|g\|_{\Phi_k}$, which leads to the desired result.

{
	\section{Wald confidence intervals} \label{app:wald}
}

As mentioned in Section \ref{sssec:binary-example}, batch means estimators may underestimate the asymptotic variance when the sample size is not sufficiently large.
Closely related to this, Wald confidence intervals may suffer from low coverage if the Monte Carlo sample size is not sufficiently large \citep{flegal2010batch,vollset1993confidence}.
In Section \ref{sssec:binary-example}, we proposed adding $(\log m) \sqrt{1/b_m^2 + b_m/m}$ to the batch means estimates of asymptotic variances, where $m$ is the Monte Carlo sample size, and $b_m$ is the batch size.
We shall apply this technique to the toy example in Section \ref{ssec:toy}.

In the toy example, for $k \in \mathsf{K}$, $\pi_k := \Pi(\{k\} \times \Z_k) = 1/k_{\scriptsize\mbox{max}}$, and $|\Z_k| = n$.
Take $k_{\scriptsize\mbox{max}} = 10$ and $n = 15$, so that $\pi_k = 0.1$ for each~$k$.
We employ the chain with slow local and global moves, which is simulated for $m=5000$ or $m = 5 \times 10^3$ iterations.
Simultaneous $1-\alpha$ Wald confidence intervals are constructed for $(\pi_1, \dots, \pi_{10})$, where $\alpha$ is 0.1 or 0.05.
Bonferroni correction is used.
The estimated asymptotic variance is either inflated by adding $(\log m) \sqrt{1/b_m^2 + b_m/m}$ or unchanged, where $b_m \approx m^{0.6}$.
The experiment is repeated 1000 times and the empirical coverage rate is compared to the nominal converge rate.
The average half-width of the $10 \times 1000$ intervals is also recorded.

The results are given in Table \ref{tab:wald}.
We see that inflating the asymptotic variance significantly enhances the empirical coverage rate when the sample size is insufficient, although the resultant intervals tend to be conservative.

\begin{table} \caption{Empirical simultaneous coverage rates of confidence intervals}\label{tab:wald}
	\centering
	\begin{tabular}{ccccc}
		\hline
		$m$ & $1-\alpha$ & inflated & empirical coverage & mean half-width \\
		\hline
		$5 \times 10^3$  & 0.90 & No & 0.796 & 0.038 \\
		$5 \times 10^3$  & 0.90 & Yes & 0.997 & 0.059 \\
		$5 \times 10^3$  & 0.95 & No & 0.842 & 0.041 \\
		$5 \times 10^3$  & 0.95 & Yes & 0.999 & 0.065 \\
		$5 \times 10^4$  & 0.90 & No & 0.888 & 0.013 \\
		$5 \times 10^4$  & 0.90 & Yes & 0.996 & 0.018 \\
		$5 \times 10^4$  & 0.95 & No & 0.935 & 0.014 \\
		$5 \times 10^4$  & 0.95 & Yes & 0.999 & 0.020 \\
		\hline
	\end{tabular}
\end{table}

\newpage

\begin{center} 
	{\bf \LARGE Supplement II: Autoregression with Laplace errors}
\end{center}

Bayesian autoregression with an unknown model order is a scenario where trans-dimensional MCMC naturally applies.
See \cite{troughton1998reversible}, \cite{ehlers2002efficient}, \cite{vermaak2004reversible}.
Consider a Bayesian autoregression with Laplace errors.
This model is practically relevant for its use in Bayesian quantile regression \citep{yu2001bayesian}.

\section{The model}

Suppose that $Y_i \in \mathbb{R}$ and $x_i \in \mathbb{R}^p$ satisfy, for $i = 1, \dots, n$, 
\[
Y_i = \sum_{j=1}^K A_j Y_{i-j} + x_i^{\top} B + \sqrt{T} E_i,
\]
where $E_1, \dots, E_n$ are iid random errors with a Laplace density $f_E(\epsilon) = e^{-|\epsilon|/2}/4$, $K$~is the unknown model order, $A_1, \dots, A_K$ are unknown autoregression coefficients,~$B$ is an unknown $p$-dimensional regression coefficient, and~$T$ is an unknown dispersion parameter.
If $K = 0$, the summation from $j = 1$ to~$K$ is interpreted as zero.
The predictors $(x_i)_{i=1}^n$ are known, while the responses $(Y_i)_{i=1}^n$ are observable.
To ensure that the model is well-specified, assume that $(Y_i)_{i=1}^n$ has a known starting sequence, $(Y_i)_{i=-k_{\scriptsize \mbox{max}} + 1}^0 = (y_i)_{i=-k_{\scriptsize \mbox{max}} + 1}^0$, where $k_{\scriptsize \mbox{max}}$ is a positive integer.
The goal is to make inference about the parameters~$K$, $(A_j)_{j=1}^K$,~$B$, and~$T$.

Let $Y = (Y_1, \dots, Y_n)^{\top} \in \mathbb{R}^n$ and $A = (A_1, \dots, A_K)^{\top} \in \mathbb{R}^K$.
Given $(K,A,B,T) = (k,\alpha, \beta, \tau)$, the likelihood of~$Y$ evaluated at $y = (y_1, \dots, y_n)^{\top}$ is
\[
f(y \mid k, \alpha, \beta, \tau) = \frac{1}{4^n \tau^{n/2}} \exp \left( - \frac{1}{2\sqrt{\tau}} \sum_{i=1}^n \left|y_i - x_i^{\top} \beta - w_{i,k}^{\top} \alpha \right| \right),
\] 
where  $w_{i,k} = (y_{i-1}, \dots, y_{i-k})^{\top}$ for $i = 1,\dots,n$.

To perform Bayesian analysis, place a prior density on $(K,A,B,T)$ of the form
\[
f_K(k) \, f_A(\alpha \mid k, \tau)^{\ind_{\{0\}^c}(k)} f_B(\beta \mid \tau) f_T(\tau).
\]
To be precise, $f_K$ is an arbitrary probability mass function that is positive on $\mathsf{K}= \{0, \dots, k_{\scriptsize \mbox{max}}\};$ 
$f_A(\cdot \mid k, \tau)$ is the density of the $\mbox{N}_k(0, \sigma^2 \tau I_k)$ distribution, where~$\sigma$ is a positive hyperparameter, and $I_k$ is the $k \times k$ identity matrix;
$f_B(\beta \mid \tau)$ is the density of the $\mbox{N}_p(0, \sigma^2 \tau I_p)$ distribution;
$f_T(\tau)$ is proportional to $\tau^{-1}$.
The resultant (un-normalized) posterior density is
\[
\tilde{\pi}(k, \alpha, \beta, \tau \mid y) = f(y \mid k, \alpha,\beta,\tau) f_K(k) \, f_A(\alpha \mid k, \tau)^{\ind_{\{0\}^c}(k)} f_B(\beta \mid \tau) f_T(\tau).
\]
The posterior distribution (if proper) is intractable even when~$k$ is given.
To facilitate sampling, we follow \cite{liu1996bayesian} and \cite{choi2013analysis}, and introduce a sequence of auxiliary random variables $U = (U_1, \dots, U_n)^{\top}$.
Given $(K,A,B,T,Y) = (k, \alpha, \beta, \tau,y)$, the conditional distribution of~$U$ is a product of inverse Gaussian distributions, with density function
\[
\begin{aligned}
	&f_{U \mid \cdot}(u \mid k, \alpha, \beta, \tau, y) \\
	=& \prod_{i=1}^n \left\{ \frac{1}{\sqrt{8\bm{\pi} u_i^3}} \exp \left[ -  \frac{u_i (y_i - x_i^{\top} \beta - w_{i,k}^{\top} \alpha )^2}{2\tau} + \frac{|y_i - x_i^{\top} \beta - w_{i,k}^{\top} \alpha|}{2\sqrt{\tau}}  - \frac{1}{8u_i} \right]  \, \ind_{(0,\infty)}(u_i) \right\},
\end{aligned}
\]
where $u = (u_1, \dots, u_n)^{\top}$, and $\bm{\pi}$ denotes the ratio of a circle's circumference to its diameter, not to be confused with the density function~$\pi$ defined below.
Consider the augmented posterior density of $(K, A, B, T, U)$, given by
\[
\pi(k,\alpha,\beta,\tau,u \mid y) = f_{U \mid \cdot}(u \mid k, \alpha, \beta, \tau, y) \, \tilde{\pi}(k, \alpha, \beta, \tau \mid y),
\]
where the argument $(k, \alpha, \beta, \tau, u)$ can take values in $\bigcup_{k \in \mathsf{K}} \{k\} \times \Z_k$, with $\Z_k = \mathbb{R}^k \times \mathbb{R}^p \times (0,\infty) \times (0,\infty)^n$.
The corresponding measure~$\Pi$ has the form~\eqref{eq:Pi}.
The map $(\alpha, \beta, \tau, u) \mapsto \pi(k, \alpha, \beta, \tau, u \mid y)$ gives the density of $\Psi_k$ with respect to the Lebesgue measure on $\Z_k$.

For $k \in \mathsf{K}$, let $W(k)$ be the $n \times (p+k)$ matrix whose $i$th row is $(x_i^{\top}, w_{i,k}^{\top})$.
In what follows, assume the following:
\begin{enumerate}
	\item [\namedlabel{P1}{(P1)}] For $k \in \mathsf{K}$, $W(k)$ has full column rank, and~$y$ is not in the column space of $W(k)$.
\end{enumerate}
It can then be shown that $\Psi_k(\Z_k) < \infty$ for $k \in \mathsf{K}$, and~$\Pi$ is a proper posterior distribution.
The introduction of the auxiliary variables $U_i$ leads to a two-component Gibbs sampler that can be used to sample from $\Phi_k$, the normalization of~$\Psi_k$. 
In what follows, the Gibbs sampler is combined with a standard reversible jump scheme to create a trans-dimensional MCMC sampler targeting~$\Pi$.

\section{A reversible jump MCMC algorithm} \label{sssec:autogibbs}

When $K$ is known to be $k \in \mathsf{K}$, a two-component Gibbs sampler can be used to sample from $\Phi_k$ \citep{choi2013analysis}.
When the current state is $(\alpha,\beta,\tau,u) \in \Z_k$, the sampler draws the next state $(\alpha',\beta',\tau',u')$ via the following steps:
\begin{enumerate}
	\item Draw $u' = (u_1', \dots,u_n')$ from the conditional distribution of $U$ given $(K,A,B,T,Y) = (k, \alpha, \beta, \tau,y)$, which is associated with the density $f_{U \mid \cdot}(u' \mid k, \alpha, \beta, \tau, y)$.
	\item 
	Let $Q $ be the $n \times n$ diagonal matrix whose $i$th diagonal element is $u_i'$.
	Draw $\tau'$ from the conditional distribution of $T$ given $(K,U,Y) = (k, u', y)$, which is the inverse Gamma distribution with shape parameter $n/2$ and scale parameter
	\[
	\frac{y^{\top} Q y - y^{\top} Q W(k) [W(k)^{\top} Q W(k) + \sigma^{-2} I_{p+k}]^{-1} W(k)^{\top} Q y }{2}.
	\]
	\item 
	Draw $(\beta'^{\top}, \alpha'^{\top})^{\top}$ from the conditional distribution of $(B^{\top}, A^{\top})^{\top}$ given $(K,T,U,Y) = (k,\tau', u', y)$, which is the normal distribution
	\[
	\mbox{N}_{p+k} \left( [W(k)^{\top} Q W(k) + \sigma^{-2} I_{p+k}]^{-1} W(k)^{\top} Q y, \, \tau [W(k)^{\top} Q W(k) + \sigma^{-2} I_{p+k}]^{-1} \right).
	\]
\end{enumerate}
Note that this is a Gibbs sampler with two components since the second and third steps are blocked, and can be viewed as drawing from the conditional distribution of $(A,B,T)$ given the rest.
{ The resultant chain has $\Phi_k$ as its stationary distribution, but is not reversible with respect to $\Phi_k$.}

Consider now a reversible jump algorithm for sampling from~$\Pi$.
When the current state is $(k,\alpha, \beta, \tau, u)$, the algorithm randomly performs one of three move types, code-named U, B, and D.
To be specific, the probability of choosing a move type depends on the current state only through~$k$, and these probabilities are denoted by $q_{\scriptsize\mbox{U}}(k)$, $q_{\scriptsize\mbox{B}}(k)$, and $q_{\scriptsize\mbox{D}}(k)$, respectively.
The three move types are defined as follows:
\begin{itemize}
	\item {\it U move}: 
	Draw $(\alpha',\beta',\tau',u')$ using one iteration of the two-component Gibbs sampler.
	Set the new state to $(k,\alpha',\beta',\tau',u')$.
	
	\item {\it B move}: 
	Draw $a_*$ from some distribution on $\mathbb{R}$ associated with a density function $g(\cdot \mid k, \alpha, \beta, \tau, u)$.
	Let $\alpha' = (\alpha^{\top}, a_*)^{\top}$.
	With probability
	\[
	\min \left\{ 1, \frac{  \pi(k+1, \alpha', \beta, \tau, u \mid y) \, q_{\scriptsize\mbox{D}}(k+1) }{  \pi(k, \alpha, \beta, \tau, u \mid y) \, q_{\scriptsize\mbox{B}}(k) \, g(a_* \mid k, \alpha, \beta, \tau, u) } \right\},
	\]
	set the new state to $(k+1,\alpha',\beta,\tau,u)$;
	otherwise, keep the old state.
	This move is only available when $k < k_{\scriptsize \mbox{max}}$.
	\item {\it D move}: 
	Let $\alpha' = (\alpha_1, \dots, \alpha_{k-1})^{\top}$ if $\alpha = (\alpha_1, \dots, \alpha_k)^{\top}$.
	With probability
	\[
	\min \left\{ 1, \frac{  \pi(k - 1, \alpha', \beta, \tau, u \mid y) \, q_{\scriptsize\mbox{B}}(k-1) \, g(\alpha_k \mid k - 1, \alpha', \beta, \tau, u, y) }{  \pi(k, \alpha, \beta, \tau, u \mid y) \, q_{\scriptsize\mbox{D}}(k) } \right\},
	\]
	set the new state to $(k-1,\alpha', \beta, \tau, u)$;
	otherwise, keep the old state.
	This move is only available when $k > 0$.
\end{itemize}

{ The resultant trans-dimensional Markov chain has $\Pi$ as its stationary distribution, but it is not reversible.}

\section{Convergence analysis}

For $k \in \mathsf{K}$, let $P_k$ be the Mtk of the Gibbs chain targeting $\Phi_k$.
Then the following result holds.

\begin{lemma} \label{lem:robustgeo}
	Suppose that \ref{P1} holds.
	For $k \in \mathsf{K}$, $P_k$ is $\Pi$-a.e. geometrically ergodic.
\end{lemma}

The proof of this result is given in Section~\ref{app:robustgeo}.
It utilizes the classical drift and minorization technique, and follows closely arguments in \cite{roy2010monte}.
\cite{roy2010monte} studied the Gibbs chain when an improper prior on $(A,B)$ is used.
The result can also be proved using the approach of \cite{choi2013analysis}, who studied the spectral properties of a Markov operator closely related to $P_k$.

Combining Lemma~\ref{lem:robustgeo} and Theorem~\ref{thm:main} yields the following for the reversible jump algorithm.

\begin{proposition} \label{pro:ar}
	Suppose that \ref{P1} holds.
	Suppose further that $q_{\scriptsize\mbox{U}}(k) > 0$ for $k \in \mathsf{K}$, $q_{\scriptsize\mbox{B}}(k) > 0$ for $k < k_{\scriptsize \mbox{max}}$, and $q_{\scriptsize\mbox{D}}(k) > 0$ for $k > 0$.
	Then the reversible jump chain is $L^2(\Pi)$ geometrically convergent and $\Pi$-a.e. geometrically ergodic.
\end{proposition}
\begin{proof}
	Apply Theorem~\ref{thm:main}.
	Let~$P$ be the Mtk of the reversible jump chain.
	For $k \in \mathsf{K}$, (i) and (iii) in \ref{H1} hold with $P_k$ defined above and $c_k = q_{\scriptsize\mbox{U}}(k)$.
	The chain associated with $P_k$ is clearly $\varphi$-irreducible.
	By Lemma~\ref{lem:robustgeo} and (ii) in Lemma~\ref{lem:Pk-convergence}, (ii) in \ref{H1} holds with $t_0 = 1$.
	
	To verify \ref{H2}, simply note that $\bar{P}(k,\{k'\})$, the average probability flow from model~$k$ to model~$k'$, is strictly positive whenever $|k'-k| \leq 1$, indicating the between-model movements are irreducible.
	(Alternatively, note that $P$ is $\Pi$-irreducible.)

	
	Thus, $P$ is $L^2(\Pi)$ geometrically convergent.
	By Theorem~1 of \cite{roberts2001geometric}, it is $\Pi$-a.e. geometrically ergodic.
\end{proof}

\section{A simulated example} \label{sssec:robustsimulate}

We use a simple simulated example to show how geometric ergodicity allows us to take estimation uncertainty into account when performing model selection.
A data set satisfying condition \ref{P1} with $k_{\scriptsize \mbox{max}} = 10$, $p = 50$, $n = 100$ is simulated.
The data set is generated according to the autoregressive model with the true model order~$K$ being~4, and the true value of~$A$ being $(0.3, 0.25, 0.05, 0.05)^{\top}$.

The reversible jump algorithm is applied to the simulated data set.
The prior distribution $f_K$ is a truncated Poisson distribution.
The density $g(\cdot \mid k, \alpha, \beta, \tau, u)$ is taken to be the conditional density function of $A_{k+1}$ given $K=k+1$, $(A_1, \dots, A_k)^{\top} = \alpha$, $B = \beta$, $T = \tau$, $U = u$, and $Y = y$.
This corresponds to a normal distribution.
Following Section 4.3 of \cite{green1995reversible}, probabilities of birth and death proposals are taken to be
\[
q_{\scriptsize\mbox{B}}(k) = \frac{1}{3} \min \left\{ 1, \frac{f_K(k+1)}{f_K(k)} \right\}, \, q_{\scriptsize\mbox{D}}(k) = \frac{1}{3} \min \left\{ 1, \frac{f_K(k-1)}{f_K(k)} \right\}.
\]
By Proposition \ref{pro:ar}, the corresponding chain is $\Pi$-a.e. geometrically ergodic.


\begin{figure}
	\centering
	\includegraphics[width=0.7\textwidth]{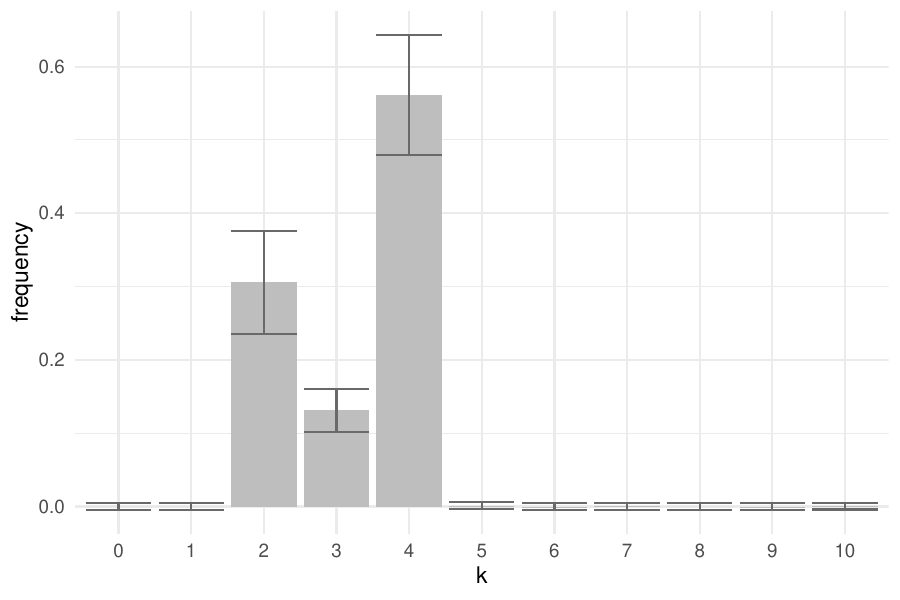}
	\caption{Estimated posterior distribution of $K$ with $95\%$ simultaneous confidence intervals.} \label{fig:varmcmc-2}
\end{figure}


Based on a simulated chain of length $m = 4 \times 10^5$, we estimate the posterior probability of $K = k$ for $k = 0,\dots,10$.
We construct simultaneous 95\% Wald confidence intervals for the posterior probabilities using Bonferroni correction.
Like in Section \ref{sssec:binary-example}, we add $(\log m) \sqrt{1/b_m^2+b_m/m}$ to the estimated asymptotic variances, where $b_m \approx m^{0.6}$ is the batch size.
The posterior probability estimates and confidence intervals are visualized in Figure \ref{fig:varmcmc-2}.
We may thus conclude with high confidence that $K=4$ is the most favored model, while $K=2$ is in second place.


\section{Proof of Lemma~\ref{lem:robustgeo}} \label{app:robustgeo}

Fix $k \in \mathcal{K}$.
Suppose that $(A(t), B(t), T(t), U(t))_{t=0}^{\infty}$ is a chain associated with $P_k$.
Then $(A(t), B(t), T(t))_{t=0}^{\infty}$ is also a Markov chain that has the same convergence rate in the total variation distance.
This is because the two chains are co-de-initializing \citep{roberts2001markov}.
Thus, it suffices to show that the latter is $\hat{\Phi}_k$-a.e. geometrically ergodic, where $\hat{\Phi}_k$ is the distribution of $(A,B,T)$ given $(K,Y) = (k,y)$.

To show geometric ergodicity one can establish a set of drift and minorization conditions \citep{rosenthal1995minorization}.
Let $\hat{P}_k$ be the Mtk of $(A(t), B(t), T(t))_{t=0}^{\infty}$.
For $(\alpha, \beta, \tau) \in \mathbb{R}^k \times \mathbb{R}^p \times (0,\infty)$, let
\[
V(\alpha, \beta, \tau) = \frac{\sum_{i=1}^N (y_i - x_i^{\top} \beta - w_{i,k}^{\top} \alpha)^2}{\tau}.
\]
The following two lemmas establish a set of drift and minorization conditions for $\hat{P}_k$.
Similar arguments can be found in, e.g., \cite{roy2010monte} and \cite{hobert2015convergence}.

\begin{lemma} \label{lem:robustdrift}
	Assume that \ref{P1} holds.
	There exist $\lambda < 1$ and $L < \infty$ such that the following holds for each $(\alpha, \beta, \tau) \in \mathbb{R}^k \times \mathbb{R}^p \times (0,\infty)$:
	\[
	\hat{P}_k V(\alpha, \beta, \tau) \leq \lambda V(\alpha, \beta, \tau) + L.
	\]
\end{lemma}

\begin{proof}
	Fix $(\alpha, \beta, \tau) \in \mathbb{R}^k \times \mathbb{R}^p \times (0,\infty)$.
	Let $(A',B',T') \sim \hat{P}_k((\alpha, \beta, \tau), \cdot)$.
	Recalling the steps of the two-component Gibbs algorithm, one can obtain
	\[
	\begin{aligned}
		\hat{P}_k V(\alpha, \beta, \tau) =& E \left[ V(A',B',T') \right] \\
		=& E\left\{ \frac{\|y - W(k) [W(k)^{\top} Q W(k) + \sigma^{-2} I_{p+k} ]^{-1} W(k)^{\top} Q y \|^2 }{T'} \right\} + \\
		& E \left( \mbox{tr} \left\{ W(k) [W(k)^{\top} Q W(k) + \sigma^{-2} I_{p+k} ]^{-1} W(k)^{\top} \right\} \right) \\
		=& E\left\{ \frac{N \| y -  W(k) [W(k)^{\top} Q W(k) + \sigma^{-2} I_{p+k} ]^{-1} W(k)^{\top} Q y \|^2 }{ y^{\top} Q y - y^{\top} Q W(k) [W(k)^{\top} Q W(k) + \sigma^{-2} I_{p+k}]^{-1} W(k)^{\top} Q y } \right\} + \\
		& E \left( \mbox{tr} \left\{ W(k) [W(k)^{\top} Q W(k) + \sigma^{-2} I_{p+k} ]^{-1} W(k)^{\top} \right\} \right),
	\end{aligned}
	\]
	where $W(k)$ is the $N \times (p+k)$ matrix whose $i$th row is $(x_i^{\top}, w_{i,k}^{\top})$, and $Q$ is the $N \times N$ diagonal matrix whose diagonal elements $(U_1',\dots, U_N')$ are distributed according to the density 
	\[
	\begin{aligned}
		&f_{U \mid \cdot}(u' \mid k, \alpha, \beta, \tau, y) \\
		=& \prod_{i=1}^N \left\{ \frac{1}{\sqrt{8\bm{\pi} u_i'^3}} \exp \left[ -  \frac{u_i' (y_i - x_i^{\top} \beta - w_{i,k}^{\top} \alpha )^2}{2\tau} + \frac{|y_i - x_i^{\top} \beta - w_{i,k}^{\top} \alpha|}{2\sqrt{\tau}}  - \frac{1}{8u_i'} \right]  \, \ind_{(0,\infty)}(u_i') \right\}.
	\end{aligned}
	\]
	Under \ref{P1}, $Q^{1/2} y$ is not in the column space of $Q^{1/2} W(k)$.
	This implies that
	\[
	\begin{aligned}
		& y^{\top} Q y - y^{\top} Q W(k) [W(k)^{\top} Q W(k) + \sigma^{-2} I_{p+k}]^{-1} W(k)^{\top} Q y \\
		\geq & y^{\top} Q y - y^{\top} Q W(k) [W(k)^{\top} Q W(k) ]^{-1} W(k)^{\top} Q y \\
		=& \left\| \left\{ I_N -  Q^{1/2} W(k) [W(k)^{\top} Q W(k)]^{-1} W(k)^{\top} Q^{1/2} \right\} Q^{1/2} y \right\|^2 \\
		>& 0.
	\end{aligned}
	\]
	Moreover,
	\[
	\begin{aligned}
		&y^{\top} Q y - y^{\top} Q W(k) [W(k)^{\top} Q W(k) + \sigma^{-2} I_{p+k}]^{-1} W(k)^{\top} Q y \\
		\geq & \left\|Q^{1/2} \left\{ y -  W(k) [W(k)^{\top} Q W(k) + \sigma^{-2} I_{p+k} ]^{-1} W(k)^{\top} Q y \right\} \right\|^2 \\
		\geq & \left(\min_i U_i' \right) \| y -  W(k) [W(k)^{\top} Q W(k) + \sigma^{-2} I_{p+k} ]^{-1} W(k)^{\top} Q y \|^2,
	\end{aligned}
	\]
	and
	\[
	\begin{aligned}
		&\mbox{tr} \left\{ W(k) [W(k)^{\top} Q W(k) + \sigma^{-2} I_{p+k} ]^{-1} W(k)^{\top} \right\} \\
		\leq & \mbox{tr} \left\{  W(k) [W(k)^{\top} Q W(k) ]^{-1} W(k)^{\top}  \right\} \\
		\leq & \left( \max_i U_i^{-1} \right) \mbox{tr} \left\{   W(k) [W(k)^{\top} W(k) ]^{-1} W(k)^{\top}  \right\} \\
		=& (p+k) \left( \max_i U_i'^{-1} \right).
	\end{aligned}
	\]
	It follows that
	\[
	\hat{P}_k V(\alpha, \beta, \tau) \leq (N+p+k) E \left( \max_i U_i'^{-1} \right) \leq (N+p+k) \sum_{i=1}^N E(U_i'^{-1}).
	\]
	Using properties of the inverse Gaussian distribution and the Cauchy-Schwarz inequality, one obtains
	\[
	\begin{aligned}
		\hat{P}_k V(\alpha, \beta, \tau) &\leq  (N+p+k) \sum_{i=1}^N \left( \frac{2 |y_i - x_i^{\top} \beta - w_{i,k}^{\top} \alpha|}{\sqrt{\tau}} + 4 \right) \\
		&\leq 2 (N+p+k) \sqrt{N} \sqrt{V(\alpha,\beta,\tau)} + 4(N+p+k).
	\end{aligned}
	\]
	The desired result follows immediately.
\end{proof}

\begin{lemma} \label{lem:robustminor}
	For each $d > 0$, there exists a nonzero measure $\nu(\cdot)$
	\[
	\hat{P}_k((\alpha, \beta, \tau), \cdot) \geq \nu(\cdot)
	\]
	whenever $V(\alpha, \beta, \tau) \leq d$.
\end{lemma}
\begin{proof}
	Denote the conditional density of $(A,B,T)$ given $(K,U,Y) = (k,u,y)$ by $f_{A,B,T \mid \cdot}(\cdot \mid k,u,y)$.
	Then the density of $\hat{P}_k((\alpha,\beta,\tau), \cdot)$ is
	\[
	\int_{(0,\infty)^N} f_{A,B,T \mid \cdot}(\cdot \mid k,y) \, f_{U \mid \cdot}(u \mid k, \alpha, \beta, \tau, u, y) \, \df u.
	\]
	Assume that $V(\alpha, \beta, \tau) \leq d$, so that $(y_i - x_i^{\top} \beta - w_{i,k}^{\top} \alpha)^2/\tau \leq d$ for $i = 1,\dots,N$.
	Then, for $u = (u_1, \dots, u_n) \in (0,\infty)^N$,
	\[
	f_{U \mid \cdot}(u \mid k, \alpha, \beta, \tau, y) \geq \prod_{i=1}^N \left[ \frac{1}{\sqrt{8\bm{\pi} u_i^3}} \exp \left( - \frac{du_i}{2} - \frac{1}{8 u_i}  \right) \right].
	\]
	The desired result is established by letting~$\nu$ be associated with the density
	\[
	\int_{(0,\infty)^N} f_{A,B,T \mid \cdot}(\cdot \mid k,u,y) \prod_{i=1}^N \left[ \frac{1}{\sqrt{8\bm{\pi} u_i^3}} \exp \left( - \frac{du_i}{2} - \frac{1}{8 u_i}  \right) \right] \, \df u.
	\]
\end{proof}

By Lemmas~\ref{lem:robustdrift} and~\ref{lem:robustminor} and the standard drift and minorization argument \citep{rosenthal1995minorization},~$\hat{P}_k$ is $\hat{\Phi}_k$-a.e. geometrically ergodic.
As previously mentioned, this implies that $P_k$ is $\Phi_k$-a.e. geometrically ergodic.

\end{document}